\documentclass[a4paper,11pt,fleqn]{article}
\usepackage{amsmath}
\usepackage{amssymb}
\usepackage{theorem}
\usepackage{charter}
\usepackage{mathrsfs} 
\usepackage{euscript}
\usepackage{xcolor}
\usepackage{enumitem}
\PassOptionsToPackage{normalem}{ulem}
\usepackage{ulem}

\renewcommand{\leq}{\ensuremath{\leqslant}}
\renewcommand{\geq}{\ensuremath{\geqslant}}
\renewcommand{\le}{\ensuremath{\leqslant}}

\newcommand{\minimize}[2]{\ensuremath{\underset{\substack{{#1}}}%
{\text{\rm minimize}}\;\;#2 }}
\newcommand{\EC}[2]{{\mathsf E}(#1\! \mid\! #2)}

\newcommand{\Frac}[2]{\displaystyle{\frac{#1}{#2}}} 
 
\newcommand{\scal}[2]{{\left\langle{{#1}\mid{#2}}\right\rangle}}

\newcommand{\menge}[2]{\big\{{#1}~\big |~{#2}\big\}} 
\newcommand{\Menge}[2]{\left\{{#1}~\Big|~{#2}\right\}} 

\newcommand{\GGG}{{\ensuremath{\boldsymbol{\mathsf G}}}}
\newcommand{\KKK}{{\ensuremath{\boldsymbol{\mathsf K}}}}
\newcommand{\HH}{\ensuremath{{\mathsf H}}}
\newcommand{\GG}{\ensuremath{{\mathsf G}}}
\newcommand{\FF}{\ensuremath{{\mathcal F}}}
\newcommand{\XX}{\ensuremath{\EuScript{X}}}

\newcommand{\WC}{\ensuremath{{\mathfrak W}}}
\newcommand{\SC}{\ensuremath{{\mathfrak S}}}
\newcommand{\XXX}{\ensuremath{\boldsymbol{\EuScript{X}}}}

\newcommand{\Sum}{\ensuremath{\displaystyle\sum}}
\newcommand{\emp}{\ensuremath{{\varnothing}}}

\newcommand{\Id}{\ensuremath{\text{\rm Id}}\,}

\newcommand{\RR}{\ensuremath{\mathbb{R}}}
\newcommand{\RP}{\ensuremath{\left[0,+\infty\right[}}

\newcommand{\BL}{\ensuremath{\EuScript B}\,}
\newcommand{\RPP}{\ensuremath{\left]0,+\infty\right[}}

\newcommand{\RX}{\ensuremath{\left]-\infty,+\infty\right]}}
\newcommand{\RXX}{\ensuremath{\left[-\infty,+\infty\right]}}
\newcommand{\EE}{\ensuremath{\mathsf E}}
\newcommand{\PP}{\ensuremath{\mathsf P}}
\newcommand{\as}{\ensuremath{\text{\rm $\PP$-a.s.}}}

\newcommand{\NN}{\ensuremath{\mathbb N}}

\newcommand{\weakly}{\ensuremath{\:\rightharpoonup\:}}

\newcommand{\pinf}{\ensuremath{{+\infty}}}

\newcommand{\dom}{\ensuremath{\text{\rm dom}\,}}
\newcommand{\prox}{\ensuremath{\text{\rm prox}}}
\newcommand{\Fix}{\ensuremath{\text{\rm Fix}\,}}

\newcommand{\gra}{\ensuremath{\text{\rm gra}\,}}

\newcommand{\infconv}{\ensuremath{\mbox{\small$\,\square\,$}}}

\newcommand{\rzeroun}{\ensuremath{\left]0,1\right]}}   
\newcommand{\argmind}[2]{\ensuremath{\underset{\substack{{#1}}}%
{\text{argmin}}\;\left(#2\right)}}

\newtheorem{theorem}{Theorem}[section]

\newtheorem{proposition}[theorem]{Proposition}
\theoremstyle{plain}{\theorembodyfont{\rmfamily}%
}
\theoremstyle{plain}{\theorembodyfont{\rmfamily}%
}
\theoremstyle{plain}{\theorembodyfont{\rmfamily}%
\newtheorem{algorithm}[theorem]{Algorithm}}
\theoremstyle{plain}{\theorembodyfont{\rmfamily}%
\newtheorem{problem}[theorem]{Problem}}
\theoremstyle{plain}{\theorembodyfont{\rmfamily}%
}
\theoremstyle{plain}{\theorembodyfont{\rmfamily}%
\newtheorem{remark}[theorem]{Remark}}
\theoremstyle{plain}{\theorembodyfont{\rmfamily}%
}
\theoremstyle{plain}{\theorembodyfont{\rmfamily}%
}

\numberwithin{equation}{section}

\begin{document}

\title{\sffamily Stochastic Approximations and Perturbations in
Forward-Backward Splitting for Monotone Operators\thanks{Contact 
author: 
P. L. Combettes, {plc@ljll.math.upmc.fr}, phone: +33 1 4427 
6319, fax: +33 1 4427 7200. This work was supported by the 
CNRS MASTODONS project under grant 2013MesureHD and by the 
CNRS Imag'in project under grant 2015OPTIMISME.}}
\author{Patrick L. Combettes$^{1}$ and 
Jean-Christophe Pesquet$^2$
\\[3mm]
\small
\small $\!^1$Sorbonne Universit\'es -- UPMC Univ. Paris 06\\
\small UMR 7598, Laboratoire Jacques-Louis Lions\\
\small F-75005 Paris, France\\
\small {plc@ljll.math.upmc.fr}
\\[3mm]
\small
\small $\!^2$Universit\'e Paris-Est\\
\small Laboratoire d'Informatique Gaspard Monge -- CNRS UMR 8049\\
\small F-77454, Marne la Vall\'ee Cedex 2, France\\
\small {jean-christophe.pesquet@univ-paris-est.fr}
}

\date{~}

\maketitle

\vskip 2mm

\begin{abstract}
\noindent

We investigate the asymptotic behavior of a stochastic version of
the forward-backward splitting algorithm for finding a zero of 
the sum of a maximally monotone set-valued operator and a 
cocoercive operator in Hilbert spaces. 
Our general setting features stochastic approximations of 
the cocoercive operator and stochastic perturbations in the 
evaluation of the resolvents of the set-valued operator. In
addition, relaxations and not necessarily vanishing proximal 
parameters are allowed. Weak and strong almost sure convergence 
properties of the iterates is
established under mild conditions on the underlying stochastic
processes. Leveraging these results, we also establish the almost 
sure convergence of the iterates of a stochastic variant of 
a primal-dual proximal splitting method for composite minimization 
problems.
\end{abstract}


\newpage
\section{Introduction}

Throughout the paper, $\HH$ is a separable real Hilbert space with 
scalar product $\scal{\cdot}{\cdot}$, associated norm 
$\|\cdot\|$, and Borel $\sigma$-algebra $\mathcal{B}$. 

A large array of problems arising in Hilbertian nonlinear analysis 
are captured by the following simple formulation.

\begin{problem}
\label{prob:1}
Let $\mathsf{A}\colon\HH\to 2^\HH$ be a set-valued maximally 
monotone operator, let $\vartheta\in\RPP$, and let 
$\mathsf{B}\colon\HH\to\HH$ be a $\vartheta$-cocoercive operator,
i.e.,
\begin{equation}
\label{e:2015-02-12a}
(\forall\mathsf{x}\in\HH)(\forall\mathsf{y}\in\HH)\quad
\scal{\mathsf{x}-\mathsf{y}}{\mathsf{B}\mathsf{x}-\mathsf{B}
\mathsf{y}}\geq\vartheta\|\mathsf{B}\mathsf{x}-
\mathsf{B}\mathsf{y}\|^2,
\end{equation}
such that
\begin{equation}
\mathsf{F}=\menge{\mathsf{z}\in\HH}{\mathsf{0}\in\mathsf{A}
\mathsf{z}+\mathsf{B}\mathsf{z}}\neq \emp.
\end{equation}
The problem is to find a point in $\mathsf{F}$.
\end{problem}

Instances of Problem~\ref{prob:1} are found in areas such as
evolution inclusions \cite{Sico10},
optimization \cite{Livre1,Lema96,Tsen90},
Nash equilibria \cite{Bric13},
image recovery \cite{Jmiv11,Chaa11,Svva10}, 
inverse problems \cite{Byrn14,Chau07},
signal processing \cite{Smms05},
statistics \cite{Devi11},
machine learning \cite{Duch09},
variational inequalities \cite{Facc03,Tsen91},
mechanics \cite{Merc79,Merc80}, 
and structure design \cite{Tali15}.
For instance, an important specialization of 
Problem~\ref{prob:1} in the
context of convex optimization is the following
\cite[Section~27.3]{Livre1}.

\begin{problem}
\label{prob:9}
Let $\mathsf{f}\colon\HH\to\RX$ be a proper lower semicontinuous 
convex function, let $\vartheta\in\RPP$, and let 
$\mathsf{g}\colon\HH\to\RR$ be a differentiable convex function 
such that $\nabla g$ is $\vartheta^{-1}$-Lipschitz 
continuous on $\HH$. The problem is to
\begin{equation}
\label{e:2015-01-20p}
\minimize{\mathsf{x}\in\HH}
{\mathsf{f}(\mathsf{x})+\mathsf{g}(\mathsf{x})},
\end{equation}
under the assumption that 
$\mathsf{F}=\text{Argmin}(\mathsf{f}+\mathsf{g})\neq\emp$.
\end{problem}

A standard method to solve Problem~\ref{prob:1} is the 
forward-backward algorithm \cite{Opti04,Lema96,Tsen91}, 
which constructs a sequence $(\mathsf{x}_n)_{n\in\NN}$ in 
$\HH$ by iterating
\begin{equation}
\label{e:2015-02-14a}
(\forall n\in\NN) \quad \mathsf{x}_{n+1} =
\mathsf{J}_{\gamma_n
\mathsf{A}}(\mathsf{x}_n-\gamma_n
\mathsf{B}\mathsf{x}_n),\quad
\text{where}\quad 0<\gamma_n<2\vartheta.
\end{equation}
Recent theoretical advances on deterministic versions of this 
algorithm can be found in \cite{Botr15,Cham15,Opti14,Yama15}.
Let us also stress that a major motivation for studying the
forward-backward algorithm is that it can be applied not only to
Problem~\ref{prob:1} \emph{per se}, but also to systems of 
coupled monotone inclusions via product space reformulations 
\cite{Sico10}, to strongly monotone composite inclusions problems 
via duality arguments \cite{Svva10,Opti14}, and to 
primal-dual composite problems via renorming in the primal-dual
space \cite{Opti14,Bang13}. Thus, new developments on 
\eqref{e:2015-02-14a} lead to new algorithms for solving
these problems.

Our paper addresses the following stochastic version of 
\eqref{e:2015-02-14a} in which, at each iteration $n$,
$u_n$ stands for a stochastic approximation to $\mathsf{B}x_n$
and $a_n$ stands for a stochastic perturbation modeling the 
approximate implementation of the resolvent operator 
$\mathsf{J}_{\gamma_n\mathsf{A}}$. 
Let $(\Omega,\FF,\PP)$ be the underlying probability space. 
An $\HH$-valued random variable is a measurable map 
$x\colon(\Omega,\FF)\to(\HH,\mathcal{B})$ and,
for every $p\in\left[1,\pinf\right[$,
$L^p(\Omega,\FF,\PP;\HH)$ denotes the space of 
equivalence classes of $\HH$-valued random variable $x$ such that 
$\int_{\Omega}\|x\|^pd\PP<\pinf$.

\begin{algorithm}
\label{algo:1}
Consider the setting of Problem~\ref{prob:1}.
Let $x_0$, $(u_n)_{n\in\NN}$, and $(a_n)_{n\in\NN}$ be random 
variables in $L^2(\Omega,\FF,\PP;\HH)$, let 
$(\lambda_n)_{n\in\NN}$ be a sequence in $\rzeroun$, and let
$(\gamma_n)_{n\in\NN}$ be a sequence in 
$\left]0,2\vartheta\right[$. Set
\begin{equation}
\label{e:FB}
(\forall n\in\NN) \quad x_{n+1} =
x_n+\lambda_n\big(\mathsf{J}_{\gamma_n \mathsf{A}}
(x_n-\gamma_nu_n)+a_n-x_n\big).
\end{equation}
\end{algorithm}

The first instances of the stochastic iteration \eqref{e:FB} can be 
traced back to \cite{Robi51} in the context of the gradient method, 
i.e., when $\mathsf{A}=\mathsf{0}$ and $\mathsf{B}$ is the 
gradient of a convex function. 
Stochastic approximations in the gradient method 
were then investigated in the Russian literature of the late 
1960s and early 1970s 
\cite{Ermo69,Ermo66,Ermo67,Guse71,Nekr74,Shor85}. 
Stochastic gradient methods have also been used extensively in 
adaptive signal processing, in control, and in machine learning,
e.g., \cite{Bach11,Kush03,Widr03}. More generally, proximal 
stochastic gradient methods have been applied to various problems; 
see for instance \cite{Atch14,Duch09,Ros14b,SS2013,Xiao14}. 

The objective of the present paper is to provide an analysis
of the stochastic forward-backward method in the context of
Algorithm~\ref{algo:1}. Almost sure convergence of the
iterates $(x_n)_{n\in\NN}$ to a solution to Problem~\ref{prob:1}
will be established under general conditions on the sequences 
$(u_n)_{n\in\NN}$, $(a_n)_{n\in\NN}$, $(\gamma_n)_{n\in\NN}$, and 
$(\lambda_n)_{n\in\NN}$. 
In particular, a feature of our analysis is that it allows for
relaxation parameters and it does 
not require that the proximal parameter sequence 
$(\gamma_n)_{n\in\NN}$ be vanishing.
Our proofs are based on properties of stochastic quasi-Fej\'er 
iterations \cite{Siop15}, for which we provide a novel convergence 
result. 

The organization of the paper is as follows. The notation is 
introduced in Section~\ref{sec:2}. Section~\ref{sec:3} provides an
asymptotic principle which will be used in Section~\ref{sec:4} to
present the main result on the weak and strong convergence of the 
iterates of Algorithm~\ref{algo:1}. Finally, Section~\ref{sec:5} 
deals with applications and proposes a stochastic primal-dual 
method.

\section{Notation}
\label{sec:2}

$\Id$ denotes the identity operator on $\HH$ and $\weakly$ and 
$\to$ denote, respectively, weak and strong convergence. The sets 
of weak and strong sequential cluster points of a sequence 
$(\mathsf{x}_n)_{n\in\NN}$ in $\HH$ are denoted by 
$\WC(\mathsf{x}_n)_{n\in\NN}$ and $\SC(\mathsf{x}_n)_{n\in\NN}$, 
respectively. 

Let $\mathsf{A}\colon\HH\to 2^{\HH}$ be a set-valued operator.
The domain of $\mathsf{A}$ is 
$\dom\mathsf{A}=\menge{\mathsf{x}\in\HH}{\mathsf{A}\mathsf{x}
\neq\emp}$ and the graph of $\mathsf{A}$ is 
$\gra\mathsf{A}=
\menge{(\mathsf{x},\mathsf{u})\in\HH\times\HH}{\mathsf{u}\in
\mathsf{A}\mathsf{x}}$.
The inverse $\mathsf{A}^{-1}$ of $\mathsf{A}$ is defined via
the equivalences $(\forall(\mathsf{x},\mathsf{u})\in\HH^2)$ 
$\mathsf{x}\in\mathsf{A}^{-1}\mathsf{u}$ $\Leftrightarrow$
$\mathsf{u}\in\mathsf{A}\mathsf{x}$.
The resolvent of $\mathsf{A}$ is 
$\mathsf{J}_\mathsf{A}=(\Id+\mathsf{A})^{-1}$.
If $\mathsf{A}$ is maximally monotone, then 
$\mathsf{J}_\mathsf{A}$ is single-valued and firmly nonexpansive,
with $\dom\mathsf{J}_\mathsf{A}=\HH$. 
An operator $\mathsf{A}\colon\HH\to 2^{\HH}$ is demiregular 
at $\mathsf{x}\in\dom\mathsf{A}$ if, for every sequence
$(\mathsf{x}_n,\mathsf{u}_n)_{n\in\NN}$ in 
$\gra\mathsf{A}$ and every $\mathsf{u}\in\mathsf{A}\mathsf{x}$ 
such that $\mathsf{x}_n\weakly\mathsf{x}$ and 
$\mathsf{u}_n\to\mathsf{u}$, we have $\mathsf{x}_n\to\mathsf{x}$
\cite{Sico10}. Let $\GG$ be a real Hilbert space.
We denote by $\BL(\HH,\GG)$ the space of bounded linear operators 
from $\HH$ to $\GG$, and we set $\BL(\HH)=\BL(\HH,\HH)$.
The adjoint of $\mathsf{L}\in \BL(\HH,\GG)$ 
is denoted by $\mathsf{L}^*$.
For more details on convex analysis and monotone 
operator theory, see \cite{Livre1}.

Let $(\Omega,\FF,\PP)$ denote the underlying probability space. 
The smallest $\sigma$-algebra generated by a family $\Phi$ of 
random variables is denoted by $\sigma(\Phi)$.
Given a sequence $(x_n)_{n\in\NN}$ of $\HH$-valued 
random variables, we denote by 
$\mathscr{X}=(\XX_n)_{n\in\NN}$ a sequence
of sigma-algebras such that
\begin{equation}
\label{e:2013-11-14}
(\forall n\in\NN)\quad\XX_n\subset\FF\quad\text{and}\quad
\sigma(x_0,\ldots,x_n)\subset\XX_n\subset\XX_{n+1}.
\end{equation}
Furthermore, we denote by $\ell_+(\mathscr{X})$ the set of sequences 
of $\RP$-valued random variables $(\xi_n)_{n\in\NN}$ such that,
for every $n\in\NN$, $\xi_n$ is $\XX_n$-measurable, and we define
\begin{equation}
\label{e:2013-11-13}
(\forall p\in\RPP)\quad\ell_+^p(\mathscr{X})=
\Menge{(\xi_n)_{n\in\NN}\in\ell_+(\mathscr{X})}
{\sum_{n\in\NN}\xi_n^p<\pinf\;\as},
\end{equation}
and
\begin{equation}
\label{e:2013-11-12}
\ell_+^\infty({\mathscr{X}})=
\Menge{(\xi_n)_{n\in\NN}\in\ell_+(\mathscr{X})}
{\sup_{n\in\NN}\xi_n<\pinf\; \as}.
\end{equation}
Equalities and inequalities involving random variables will always
be understood to hold $\PP$-almost surely, although this will not
always be expressly mentioned. 
Let $\mathcal{E}$ be a sub sigma-algebra of $\FF$,
let $x\in L^1(\Omega,\FF,\PP;\HH)$, and let 
$y\in L^1(\Omega,\mathcal{E},\PP;\HH)$. Then $y$ is the
conditional expectation of $x$ with respect to $\mathcal{E}$ if
$(\forall E\in\mathcal{E})$ $\int_Exd\PP=\int_Eyd\PP$; in this case
we write $y=\EC{x}{\mathcal{E}}$. We have
\begin{equation}
\label{e:2015-01-21}
\big(\forall x\in L^1(\Omega,\FF,\PP;\HH)\big)\quad
\|\EC{x}{\mathcal{E}}\|\leq\EC{\|x\|}{\mathcal{E}}.
\end{equation}
In addition,
$L^2(\Omega,\FF,\PP;\HH)$ is a Hilbert space and
\begin{equation}
\label{e:2015-01-22}
\big(\forall x\in L^2(\Omega,\FF,\PP;\HH)\big)\quad
\begin{cases}
\|\EC{x}{\mathcal{E}}\|^2\leq\EC{\|x\|^2}{\mathcal{E}}\\
(\forall\mathsf{u}\in\HH)\quad
\EC{\scal{x}{\mathsf{u}}}{\mathcal{E}}=
\scal{\EC{x}{\mathcal{E}}}{\mathsf{u}}.
\end{cases}
\end{equation}
Geometrically, if $x\in L^2(\Omega,\FF,\PP;\HH)$, 
$\EC{x}{\mathcal{E}}$ is the projection of $x$ onto
$L^2(\Omega,\mathcal{E},\PP;\HH)$.
For background on probability in Hilbert spaces, see 
\cite{Fort95,Ledo91}.

\section{An asymptotic principle}
\label{sec:3}

In this section, we establish an asymptotic principle which will
lay the foundation for the convergence analysis of our stochastic
forward-backward algorithm. First, we need the following result.

\begin{proposition}
\label{p:1}
Let $\mathsf{F}$ be a nonempty closed subset of $\HH$, let 
$\phi\colon\RP\to\RP$ be a strictly increasing 
function such that $\lim_{t\to\pinf}\phi(t)=\pinf$, let 
$(x_n)_{n\in\NN}$ be a sequence of $\HH$-valued random variables,
and let $(\XX_n)_{n\in\NN}$ be a sequence
of sub-sigma-algebras of $\FF$ such that
\begin{equation}
\label{e:2013-11-14'}
(\forall n\in\NN)\quad\sigma(x_0,\ldots,x_n)\subset
\XX_n\subset\XX_{n+1}.
\end{equation}
Suppose that, for every $\mathsf{z}\in\mathsf{F}$, there exist
$(\vartheta_n(\mathsf{z}))_{n\in\NN}\in\ell_+({\mathscr{X}})$,
$(\chi_n(\mathsf{z}))_{n\in\NN}\in\ell_+^1({\mathscr{X}})$,
and $(\eta_n(\mathsf{z}))_{n\in\NN}\in\ell_+^1({\mathscr{X}})$
such that
\begin{equation}
\label{e:sqf1}
(\forall n\in\NN)\quad
\EC{\phi(\|x_{n+1}-\mathsf{z}\|)}{\XX_n}+\vartheta_n(\mathsf{z}) 
\leq (1+\chi_n(\mathsf{z}))\phi(\|x_n-\mathsf{z}\|)+
\eta_n(\mathsf{z})\;\as
\end{equation}
Then the following hold:
\begin{enumerate}
\item
\label{p:1i}
$(\forall\mathsf{z}\in\mathsf{F})$ 
$\big[\:\sum_{n\in\NN}\vartheta_n(\mathsf{z})<\pinf\:\as\big]$
\item
\label{p:1ii}
$(x_n)_{n\in\NN}$ is bounded $\as$
\item
\label{p:1iibis}
There exists $\widetilde{\Omega}\in\FF$ such that 
$\PP(\widetilde{\Omega})=1$ and, for every 
$\omega\in\widetilde{\Omega}$ and every $\mathsf{z}\in\mathsf{F}$,
$(\|x_n(\omega)-\mathsf{z}\|)_{n\in\NN}$ converges.
\item
\label{p:1iii}
Suppose that $\WC(x_n)_{n\in\NN}\subset\mathsf{F}\;\:\as$ Then 
$(x_n)_{n\in\NN}$ converges weakly $\as$ to an $\mathsf{F}$-valued 
random variable.
\item
\label{p:1iv}
Suppose that $\SC(x_n)_{n\in\NN}\cap\mathsf{F}\neq\emp\;\:\as$
Then $(x_n)_{n\in\NN}$ converges strongly $\as$ to an 
$\mathsf{F}$-valued random variable.
\item
\label{p:1v}
Suppose that $\SC(x_n)_{n\in\NN}\neq\emp\;\:\as$ and that 
$\WC(x_n)_{n\in\NN}\subset\mathsf{F}\;\:\as$ Then 
$(x_n)_{n\in\NN}$ converges strongly $\as$ to an $\mathsf{F}$-valued 
random variable.
\end{enumerate}
\end{proposition}
\begin{proof}
This is \cite[Proposition~2.3]{Siop15} in the case when
$(\forall n\in\NN)$ $\XX_n=\sigma(x_0,\ldots,x_n)$. However, the
proof remains the same in the more general setting of
\eqref{e:2013-11-14}.
\end{proof}

The following result describes the asymptotic behavior of an
abstract stochastic recursion in Hilbert spaces.

\begin{theorem}
\label{t:1}
Let ${\mathsf F}$ be a nonempty closed subset of $\HH$ and let 
$(\lambda_n)_{n\in\NN}$ be a sequence in $\rzeroun$. 
In addition, let $(x_n)_{n\in\NN}$, $(t_n)_{n\in\NN}$, 
$(c_n)_{n\in\NN}$, and $(d_n)_{n\in\NN}$ be sequences in
$L^2(\Omega,\FF,\PP;\HH)$. Suppose that the following 
are satisfied:
\begin{enumerate}[label=\rm(\alph*)]
\item
\label{a:1i-}
$\mathscr{X}=(\XX_n)_{n\in\NN}$ is a sequence of 
sub-sigma-algebras of $\FF$ 
such that $(\forall n\!\in\!\NN)$ $\sigma(x_0,\ldots,x_n)\subset
\XX_n\subset\XX_{n+1}$.
\item
\label{a:1i}
$(\forall n\in\NN)$ $x_{n+1}=x_n+\lambda_n(t_n+c_n-x_n)$.
\item
\label{a:1ii}
$\sum_{n\in\NN}\lambda_n\sqrt{\EC{\|c_n\|^2}{\XX_n}}<\pinf$ and 
$\sum_{n\in\NN}\sqrt{\lambda_n\EC{\|d_n\|^2}{\XX_n}}<\pinf$.
\item
\label{a:1iii}
For every $\mathsf{z}\in\mathsf{F}$, there exist a sequence
$(s_n(\mathsf{z}))_{n\in\NN}$ of $\HH$-valued random variables,
$(\theta_{1,n}(\mathsf{z}))_{n\in\NN}\in\ell_+({\mathscr{X}})$, 
$(\theta_{2,n}(\mathsf{z}))_{n\in\NN}\in\ell_+({\mathscr{X}})$, 
$(\mu_{1,n}(\mathsf{z}))_{n\in\NN}\in\ell_+^\infty({\mathscr{X}})$,
$(\mu_{2,n}(\mathsf{z}))_{n\in\NN}\in\ell_+^\infty({\mathscr{X}})$,
$(\nu_{1,n}(\mathsf{z}))_{n\in\NN}\in\ell_+^\infty({\mathscr{X}})$,
and 
$(\nu_{2,n}(\mathsf{z}))_{n\in\NN}\in\ell_+^\infty({\mathscr{X}})$
such that 
$(\lambda_n\mu_{1,n}(\mathsf{z}))_{n\in\NN}\in
\ell_+^1({\mathscr{X}})$,
$(\lambda_n\mu_{2,n}(\mathsf{z}))_{n\in\NN}\in
\ell_+^1({\mathscr{X}})$,
$(\lambda_n\nu_{1,n}(\mathsf{z}))_{n\in\NN}\in\ell_+^{1/2}
({\mathscr{X}})$, 
$(\lambda_n\nu_{2,n}(\mathsf{z}))_{n\in\NN}\in\ell_+^{1/2}
({\mathscr{X}})$, 
\begin{equation}
\label{e:condi1}
(\forall n\in\NN)\;\;
\EC{\|t_n-\mathsf{z}\|^2}{\XX_n} +\theta_{1,n}(\mathsf{z})\le
(1+\mu_{1,n}(\mathsf{z})) \EC{\|s_n(\mathsf{z})+d_n\|^2}{\XX_n}+
\nu_{1,n}(\mathsf{z}),
\end{equation}
and
\begin{equation}
\label{e:condi2}
(\forall n\in\NN)\;\;
\EC{\|s_n(\mathsf{z})\|^2}{\XX_n}+ \theta_{2,n}(\mathsf{z})\leq
(1+\mu_{2,n}(\mathsf{z}))\|x_n-\mathsf{z}\|^2+\nu_{2,n}(\mathsf{z}).
\end{equation}
\end{enumerate}
Then the following hold:
\begin{enumerate}
\item
\label{t:1i}
$(\forall\mathsf{z}\in\mathsf{F})$
$\big[\:\sum_{n\in\NN} \lambda_n \theta_{1,n}(\mathsf{z}) <\pinf\;
\;\text{and}\;\sum_{n\in\NN} \lambda_n 
\theta_{2,n}(\mathsf{z})<\pinf\quad\as\:\big]$.
\item 
\label{t:1ii}
$\sum_{n\in\NN}\lambda_n(1-\lambda_n)\EC{\|t_n-x_n\|^2}{\XX_n}
<\pinf\;\as$
\item 
\label{t:1iii}
Suppose that $\WC(x_n)_{n\in\NN}\subset\mathsf{F}\;\:\as$ 
Then $(x_n)_{n\in\NN}$ converges weakly $\as$ to an 
$\mathsf{F}$-valued random variable.
\item
\label{t:1iv}
Suppose that $\SC(x_n)_{n\in\NN}\cap\mathsf{F}\neq\emp\;\:\as$
Then $(x_n)_{n\in\NN}$ converges strongly $\as$ to an 
$\mathsf{F}$-valued random variable.
\item
\label{t:1v}
Suppose that $\SC(x_n)_{n\in\NN}\neq\emp\;\:\as$ and that 
$\WC(x_n)_{n\in\NN}\subset\mathsf{F}\;\:\as$ Then 
$(x_n)_{n\in\NN}$ converges strongly $\as$ to an $\mathsf{F}$-valued 
random variable.
\end{enumerate}
\end{theorem}
\begin{proof}
Let $\mathsf{z}\in\mathsf{F}$. By \eqref{e:2015-01-22} and 
\eqref{e:condi1},
\begin{align}
\label{e:ntnzb}
(\forall n\in\NN)\quad
\EC{\|t_n-\mathsf{z}\|}{\XX_n} 
&\leq\sqrt{\EC{\|t_n-\mathsf{z}\|^2}{\XX_n}}\nonumber\\ 
&\leq\sqrt{1+\mu_{1,n}(\mathsf{z})}
\sqrt{\EC{\|s_n(\mathsf{z})+d_n\|^2}{\XX_n}}+\sqrt{\nu_{1,n}
(\mathsf{z})}\nonumber\\
&\leq\Big(1+\frac{\mu_{1,n}(\mathsf{z})}{2}\Big)
\sqrt{\EC{\|s_n(\mathsf{z})+d_n\|^2}{\XX_n}}+
\sqrt{\nu_{1,n}(\mathsf{z})}.
\end{align}
On the other hand, according to the triangle inequality and
\eqref{e:condi2}, 
\begin{align}
\label{e:ntnzb2}
(\forall n\in\NN)\quad
\sqrt{\EC{\|s_n(\mathsf{z})+d_n\|^2}{\XX_n}}
&\leq\sqrt{\EC{\|s_n(\mathsf{z})\|^2}{\XX_n}}
+\sqrt{\EC{\|d_n\|^2}{\XX_n}}\nonumber\\
&\leq\sqrt{1+\mu_{2,n}(\mathsf{z})}\|x_n-\mathsf{z}\|
+\sqrt{\nu_{2,n}(\mathsf{z})}+\sqrt{\EC{\|d_n\|^2}{\XX_n}}\nonumber\\
&\leq\Big(1+\frac{\mu_{2,n}(\mathsf{z})}{2}\Big)\|x_n-\mathsf{z}\|
+\sqrt{\nu_{2,n}(\mathsf{z})}+\sqrt{\EC{\|d_n\|^2}{\XX_n}}.
\end{align}
Furthermore, \ref{a:1i} yields
\begin{equation}
(\forall n\in\NN)\quad
\|x_{n+1}-\mathsf{z}\|\le(1-\lambda_n)\|x_n-\mathsf{z}\|+
\lambda_n\|t_n-\mathsf{z}\|+\lambda_n\|c_n\|.
\end{equation}
Consequently, \eqref{e:ntnzb} and \eqref{e:ntnzb2} lead to
\begin{align}
\label{e:Rob1}
(\forall n\in\NN)\quad
\EC{\|x_{n+1}-\mathsf{z}\|}{\XX_n}
&\leq\; (1-\lambda_n) \|x_n-\mathsf{z}\|+\lambda_n
\EC{\|t_n-\mathsf{z}\|}{\XX_n} +\lambda_n
\EC{\|c_n\|}{\XX_n}\nonumber\\
&\leq\;(1+\rho_n(\mathsf{z})) \|x_n-\mathsf{z}\|+
\zeta_n(\mathsf{z}),
\end{align}
where
\begin{equation}
\rho_n(\mathsf{z}) =
\frac{\lambda_n}{2}\bigg(\mu_{1,n}(\mathsf{z})+\mu_{2,n}
(\mathsf{z})+\frac{\mu_{1,n}(\mathsf{z})\mu_{2,n}
(\mathsf{z})}{2}\bigg)
\end{equation}
and
\begin{equation}
\zeta_n(\mathsf{z})=\lambda_n
\sqrt{\nu_{1,n}(\mathsf{z})}+\lambda_n
\Big(1+\frac{\mu_{1,n}(\mathsf{z})}{2}\Big) 
\Big(\sqrt{\nu_{2,n}(\mathsf{z})}+
\sqrt{\EC{\|d_n\|^2}{\XX_n}}\Big)+\lambda_n\EC{\|c_n\|}{\XX_n}.
\end{equation}
Now set
\begin{equation}
\overline{\mu}_1(\mathsf{z})=\sup_{n\in\NN}\mu_{1,n}(\mathsf{z}).
\end{equation}
In view of \eqref{e:condi1} and \eqref{e:condi2}, we have
\begin{align}
\label{e:Rob1f}
2\sum_{n\in\NN} \rho_n(\mathsf{z})
&=\sum_{n\in\NN} \lambda_n\mu_{1,n}(\mathsf{z})+\sum_{n\in\NN}
\lambda_n\mu_{2,n}(\mathsf{z})+\frac12\sum_{n\in\NN}
\lambda_n\mu_{1,n}(\mathsf{z})\mu_{2,n}(\mathsf{z})\nonumber\\
&\leq  \sum_{n\in\NN} \lambda_n\mu_{1,n}(\mathsf{z})+\Big(1+
\frac{\overline{\mu}_1(\mathsf{z})}{2}\Big) \sum_{n\in\NN}
\lambda_n\mu_{2,n}(\mathsf{z})\nonumber\\
&<\pinf.
\end{align}
In addition, since \eqref{e:2015-01-22} yields
\begin{equation}
\label{e:l1absen}
(\forall n\in\NN) \quad
\EC{\|c_n\|}{\XX_n}\leq \sqrt{\EC{\|c_n\|^2}{\XX_n}},
\end{equation}
we derive from \ref{a:1ii} and \ref{a:1iii} that
\begin{align}
\label{e:l1absenp}
\sum_{n\in\NN} \zeta_n(\mathsf{z}) 
&\leq\sum_{n\in\NN} \sqrt{\lambda_n\nu_{1,n}(\mathsf{z})} +
\bigg(1+\frac{\overline{\mu}_1(\mathsf{z})}{2}\bigg)
\bigg(\sum_{n\in\NN}\sqrt{\lambda_n\nu_{2,n}(\mathsf{z})}
+\sum_{n\in\NN}\sqrt{\lambda_n\EC{\|d_n\|^2}{\XX_n}}\bigg)
\nonumber\\
&\quad\;+\sum_{n\in\NN}\lambda_n\sqrt{\EC{\|c_n\|^2}{\XX_n}}
\nonumber\\
&<\pinf. 
\end{align}
Using Proposition~\ref{p:1}\ref{p:1ii}, \eqref{e:Rob1}, 
\eqref{e:Rob1f}, and \eqref{e:l1absenp}, we obtain that
\begin{equation}
\label{e:2015-02-17a}
\big(\|x_n-\mathsf{z}\|\big)_{n\in\NN}\:\;\text{is almost 
surely bounded}.
\end{equation}
In turn, by \eqref{e:condi2},
\begin{equation}
\label{e:2015-02-17b}
\big(\EC{\|s_n(\mathsf{z})\|^2}{\XX_n}\big)_{n\in\NN}\:\;
\text{is almost surely bounded}.
\end{equation}
In addition, \eqref{e:condi1} implies that
\begin{equation}
(\forall n\in\NN)\quad
\EC{\|t_n-\mathsf{z}\|^2}{\XX_n}\leq 2(1+\overline{\mu}_{1}
(\mathsf{z}))\big(\EC{\|s_n(\mathsf{z})\|^2}{\XX_n}+
\EC{\|d_n\|^2}{\XX_n}\big)+\nu_{1,n}(\mathsf{z}),
\end{equation}
from which we deduce that 
\begin{equation}
\label{e:2015-02-17c}
\big(\lambda_n\EC{\|t_n-\mathsf{z}\|^2}{\XX_n}\big)_{n\in\NN}\:\;
\text{is almost surely bounded}.
\end{equation}
Next, we observe that \eqref{e:condi1} and \eqref{e:condi2} yield
\begin{align}
\label{e:Qr6u5-O-27c}
(\forall n\in\NN)\quad
&\EC{\|t_n-\mathsf{z}\|^2}{\XX_n}+\theta_{1,n}(\mathsf{z})+
(1+\mu_{1,n}(\mathsf{z}))\theta_{2,n}(\mathsf{z})\nonumber\\
&\leq (1+\mu_{1,n}(\mathsf{z}))(1+\mu_{2,n}(\mathsf{z}))
\|x_n-\mathsf{z}\|^2+\nu_{1,n}(\mathsf{z})\nonumber\\
&\quad\; +(1+\mu_{1,n}(\mathsf{z}))
\big(\nu_{2,n}(\mathsf{z})+
2\EC{\scal{s_n(\mathsf{z})}{d_n}}{\XX_n}+\EC{\|d_n\|^2}{\XX_n}\big).
\end{align}
Now set
\begin{equation}
\label{e:blablaR}
\begin{cases}
\theta_n(\mathsf{z})=\theta_{1,n}(\mathsf{z})+
(1+\mu_{1,n}(\mathsf{z}))\theta_{2,n}(\mathsf{z})\\
\mu_n(\mathsf{z})=\mu_{1,n}(\mathsf{z})+
(1+\overline{\mu}_1(\mathsf{z}))\mu_{2,n}(\mathsf{z})\\
\nu_n(\mathsf{z})=\nu_{1,n}(\mathsf{z})+
(1+\overline{\mu}_1(\mathsf{z}))
\big(\nu_{2,n}(\mathsf{z})+2\sqrt{\EC{\|s_n(\mathsf{z})\|^2}{\XX_n}}
\sqrt{\EC{\|d_n\|^2}{\XX_n}}+\EC{\|d_n\|^2}{\XX_n}\big)\\
\xi_n(\mathsf{z})=2\lambda_n\|t_n-\mathsf{z}\|\,\|c_n\|+
2(1-\lambda_n)\|x_n-\mathsf{z}\|\,\|c_n\|+\lambda_n\|c_n\|^2.
\end{cases}
\end{equation}
By the Cauchy-Schwarz inequality and \eqref{e:Qr6u5-O-27c}, 
\begin{equation}\label{e:ECtzkexz}
(\forall n\in\NN)\quad
\EC{\|t_n-\mathsf{z}\|^2}{\XX_n}+\theta_n(\mathsf{z})\leq
(1+\mu_n(\mathsf{z}))\|x_n-\mathsf{z}\|^2+\nu_n(\mathsf{z}).
\end{equation}
On the other hand, by the conditional Cauchy-Schwarz inequality,
\begin{align}
(\forall n\in\NN)\quad 
\lambda_n\EC{\xi_n(\mathsf{z})}{\XX_n}
&\leq 2(1-\lambda_n)\lambda_n\|x_n-\mathsf{z}\|
\,\EC{\|c_n\|}{\XX_n}\nonumber\\
&\quad\;+2\lambda_n\sqrt{\lambda_n\EC{\|t_n-\mathsf{z}\|^2}{\XX_n}}
\sqrt{\lambda_n\EC{\|c_n\|^2}{\XX_n}}+\lambda_n^2
\EC{\|c_n\|^2}{\XX_n}\nonumber\\
&\leq 2\|x_n-\mathsf{z}\|
\,\lambda_n\sqrt{\EC{\|c_n\|^2}{\XX_n}}\nonumber\\
&\quad\;+2\sqrt{\lambda_n\EC{\|t_n-\mathsf{z}\|^2}{\XX_n}}
\lambda_n\sqrt{\EC{\|c_n\|^2}{\XX_n}}+
\lambda_n^2\EC{\|c_n\|^2}{\XX_n}. 
\end{align}
Thus, it follows from \eqref{e:2015-02-17a}, \ref{a:1ii}, and 
\eqref{e:2015-02-17c} that 
\begin{equation}
\label{e:2015-02-17d}
\sum_{n\in\NN}\lambda_n\EC{\xi_n(\mathsf{z})}{\XX_n}<\pinf.
\end{equation}
Let us define
\begin{equation}
\label{e:2013-11-15a}
(\forall n\in\NN)\quad 
\begin{cases}
\vartheta_n(\mathsf{z})=\lambda_n\theta_n(\mathsf{z})+
\lambda_n(1-\lambda_n)\EC{\|t_n-x_n\|^2}{\XX_n}\\
\chi_n(\mathsf{z})=\lambda_n\mu_n(\mathsf{z})\\
\eta_n(\mathsf{z})=\lambda_n\EC{\xi_n(\mathsf{z})}{\XX_n}+
\lambda_n\nu_n(\mathsf{z}).
\end{cases}
\end{equation}
It follows from \ref{a:1ii}, \ref{a:1iii}, 
\eqref{e:2015-02-17b}, and the inclusion
$\ell_+^{1/2}({\mathscr{X}})\subset\ell_+^{1}({\mathscr{X}})$ that
$(\theta_n(\mathsf{z}))_{n\in\NN}\in\ell_+({\mathscr{X}})$, 
$(\lambda_n\mu_n(\mathsf{z}))_{n\in\NN}\in\ell_+^1({\mathscr{X}})$, 
and
$(\lambda_n\nu_n(\mathsf{z}))_{n\in\NN}\in\ell_+^{1}({\mathscr{X}})$.
Therefore,  
\begin{equation}
\label{e:Huy9i-22a}
\big(\vartheta_n(\mathsf{z})\big)_{n\in\NN}\in\ell_+({\mathscr{X}})
\end{equation}
and
\begin{equation}
\label{e:Huy9i-22b}
\big(\chi_n(\mathsf{z})\big)_{n\in\NN}\in\ell_+^1({\mathscr{X}}).
\end{equation}
Furthermore, we deduce from \eqref{e:2015-02-17d} that
\begin{equation}
\label{e:Huy9i-22c}
\big(\eta_n(\mathsf{z})\big)_{n\in\NN}\in\ell_+^1({\mathscr{X}}).
\end{equation}
Next, we derive from \ref{a:1i}, \cite[Corollary~2.14]{Livre1}, and 
\eqref{e:ECtzkexz} that
\begin{align}
\label{e:2013-11-16y}
(\forall n\in\NN)\quad
\EC{\|x_{n+1}-\mathsf{z}\|^2}{\XX_n}
&=\EC{\|(1-\lambda_n)(x_n -\mathsf{z})+\lambda_n(
t_n-\mathsf{z}+c_n)\|^2}{\XX_n}\nonumber\\
&=(1-\lambda_n)\EC{\|x_n -\mathsf{z}\|^2}{\XX_n}+
\lambda_n\EC{\|t_n-\mathsf{z}+c_n\|^2}{\XX_n}\nonumber\\
&\quad\;-\lambda_n(1-\lambda_n)\EC{\|t_n-x_n+c_n\|^2}{\XX_n}
\nonumber\\
&=(1-\lambda_n)\|x_n-\mathsf{z}\|^2+
\lambda_n\EC{\|t_n-\mathsf{z}\|^2}{\XX_n}\nonumber\\
&\quad\;+2\lambda_n\EC{\scal{t_n-\mathsf{z}}{c_n}}{\XX_n}
-\lambda_n(1-\lambda_n)\EC{\|t_n-x_n\|^2}{\XX_n}\nonumber\\
&\quad\;
-2\lambda_n(1-\lambda_n)\EC{\scal{t_n-x_n}{c_n}}{\XX_n}
+\lambda_n^2\EC{\|c_n\|^2}{\XX_n}
\nonumber\\
&=(1-\lambda_n)\|x_n-\mathsf{z}\|^2+
\lambda_n\EC{\|t_n-\mathsf{z}\|^2}{\XX_n}\nonumber\\
&\quad\;-\lambda_n(1-\lambda_n)\EC{\|t_n-x_n\|^2}{\XX_n}
+2\lambda_n^2\EC{\scal{t_n-\mathsf{z}}{c_n}}{\XX_n}
\nonumber\\
&\quad\;
+2\lambda_n(1-\lambda_n)\EC{\scal{x_n-z}{c_n}}{\XX_n}
+\lambda_n^2\EC{\|c_n\|^2}{\XX_n}
\nonumber\\
&\leq(1-\lambda_n)\|x_n-\mathsf{z}\|^2+
\lambda_n\EC{\|t_n-\mathsf{z}\|^2}{\XX_n}\nonumber\\
&\quad\;
-\lambda_n(1-\lambda_n)\EC{\|t_n-x_n\|^2}{\XX_n}
+\lambda_n\EC{\xi_n(\mathsf{z})}{\XX_n}\nonumber\\
&\leq\big(1+\chi_n(\mathsf{z})\big)\|x_n-\mathsf{z}\|^2-
\vartheta_n(\mathsf{z})+\eta_n(\mathsf{z}).
\end{align}
We therefore recover \eqref{e:sqf1} with $\phi\colon t\mapsto t^2$.
Hence, appealing to \eqref{e:Huy9i-22a}, \eqref{e:Huy9i-22b},
\eqref{e:Huy9i-22c}, and Proposition~\ref{p:1}\ref{p:1i}, we 
obtain 
$(\vartheta_n(\mathsf{z}))_{n\in\NN}\in\ell_+^1({\mathscr{X}})$,
which establishes \ref{t:1i} and \ref{t:1ii}.
Finally, \ref{t:1iii}--\ref{t:1v} follow from 
Proposition~\ref{p:1}\ref{p:1iii}--\ref{p:1v}.
\end{proof}

\begin{remark}\rm\
\begin{enumerate}
\item Theorem~\ref{t:1} extends \cite[Theorem~2.5]{Siop15}, which 
corresponds to the special case when, for every $n\in\NN$ and 
every $\mathsf{z}\in\mathsf{F}$, 
$\mu_{1,n}(\mathsf{z})=\nu_{1,n}(\mathsf{z})=
\theta_{2,n}(\mathsf{z})=0$ and $d_n=0$.
Note that the $L^2$ assumptions in Theorem~\ref{t:1} are just made
to unify the presentation with the forthcoming results of
Section~\ref{sec:4}. However, since we take only conditional
expectations of $\RP$-valued random variables, they are not 
necessary.
\item 
Suppose that $(\forall n\in\NN)$ $c_n=d_n=0$. Then 
\eqref{e:blablaR} and \eqref{e:2013-11-15a} imply that
\begin{equation}
(\forall n\in\NN)\quad 
\eta_n(\mathsf{z})=\lambda_n
\big(\nu_{1,n}(\mathsf{z})+
(1+\overline{\mu}_1(\mathsf{z})) \nu_{2,n}(\mathsf{z})\big),
\end{equation}
and it follows directly from 
\eqref{e:2013-11-16y} and Proposition \ref{p:1} that
the conditions on $(\nu_{1,n}(\mathsf{z}))_{n\in\NN}$
and $(\nu_{1,n}(\mathsf{z}))_{n\in\NN}$ can be weakened to
$(\lambda_n\nu_{1,n}(\mathsf{z}))_{n\in\NN}\in
\ell_+^1({\mathscr{X}})$ and
$(\lambda_n\nu_{2,n}(\mathsf{z}))_{n\in\NN}\in
\ell_+^1({\mathscr{X}})$.
\end{enumerate}
\end{remark}

\section{A stochastic forward-backward algorithm}
\label{sec:4}

We now state our the main result of the paper.
\begin{theorem}
\label{t:2}
Consider the setting of Problem~\ref{prob:1}, let
$(\tau_n)_{n\in\NN}$ be a sequence in $\RP$, 
let $\mathscr{X}=(\XX_n)_{n\in\NN}$ be a sequence of 
sub-sigma-algebras of $\FF$, and let
$(x_n)_{n\in\NN}$ be a sequence generated by 
Algorithm~\ref{algo:1}. Assume that the following are satisfied:
\begin{enumerate}[label=\rm(\alph*)]
\item
\label{a:t2i-}
$(\forall n\!\in\!\NN)$ $\sigma(x_0,\ldots,x_n)\subset
\XX_n\subset\XX_{n+1}$.
\item 
\label{a:t2i}
$\sum_{n\in\NN}\lambda_n\sqrt{\EC{\|a_n\|^2}{\XX_n}}<\pinf$.
\item 
\label{a:t2ii}
$\sum_{n\in\NN}\sqrt{\lambda_n}
\|\EC{u_n}{\XX_n}-\mathsf{B}x_n\|<\pinf$.
\item 
\label{a:t2iii} 
For every $\mathsf{z}\in\mathsf{F}$, there exists 
$\big(\zeta_n(\mathsf{z})\big)_{n\in\NN}\in
\ell^\infty_+({\mathscr{X}})$ such that
$\big(\lambda_n\zeta_n(\mathsf{z})\big)_{n\in\NN}\in
\ell_+^{1/2}({\mathscr{X}})$ and
\begin{equation}
\label{e:boundtauetazeta}
(\forall n\in\NN)\quad 
\EC{\|u_n-\EC{u_n}{\XX_n}\|^2}{\XX_n}\leq\tau_n\|\mathsf{B} x_n
-\mathsf{B}\mathsf{z}\|^2+\zeta_n(\mathsf{z}).
\end{equation}
\item 
\label{a:t2iv} 
$\inf_{n\in\NN}\gamma_n>0$, $\sup_{n\in\NN}\tau_n<\pinf$, and
$\sup_{n\in\NN}(1+\tau_n)\gamma_n<2\vartheta$.
\item 
\label{a:t2v} 
Either $\inf_{n\in\NN}\lambda_n>0$ or
$\big[\,\gamma_n\equiv\gamma$, $\sum_{n\in\NN}\tau_n<\pinf$, and
$\sum_{n\in\NN}\lambda_n=\pinf\,\big]$.
\end{enumerate}
Then the following hold for some $\mathsf{F}$-valued random 
variable $x$:
\begin{enumerate}
\item
\label{t:2i}
$\sum_{n\in\NN}\lambda_n\|\mathsf{B}x_n-\mathsf{B}\mathsf{z}\|^2
<\pinf\;\as$
\item
\label{t:2ii}
$\sum_{n\in\NN}\lambda_n\|x_n-\gamma_n\mathsf{B}x_n-
\mathsf{J}_{\gamma_n\mathsf{A}}(x_n-\gamma_n\mathsf{B}x_n)
+\gamma_n\mathsf{B}\mathsf{z}\|^2<\pinf\;\as$
\item
\label{t:2iii}
$(x_n)_{n\in\NN}$ converges weakly $\as$ to $x$.
\item
\label{t:2iv}
Suppose that one of the following is satisfied:
\begin{enumerate}[label=\rm(\alph*)]
\setcounter{enumii}{6}
\item
\label{a:nyc2014-04-03v}
$\mathsf{A}$ is demiregular at every $\mathsf{z}\in\mathsf{F}$.
\item
\label{a:nyc2014-04-03vi}
$\mathsf{B}$ is demiregular at every $\mathsf{z}\in\mathsf{F}$.
\end{enumerate}
Then $(x_n)_{n\in\NN}$ converges strongly $\as$ to $x$.
\end{enumerate}
\end{theorem}
\begin{proof}
Set
\begin{equation}
\label{e:es175}
(\forall n\in\NN)\quad\mathsf{R}_n=\Id -\gamma_n \mathsf{B},\;
r_n=x_n-\gamma_n u_n,\;\text{and}\;
t_n=\mathsf{J}_{\gamma_n \mathsf{A}}r_n.
\end{equation}
Then it follows from \eqref{e:FB} that assumption~\ref{a:1i} in 
Theorem~\ref{t:1} is satisfied with 
\begin{equation}
\label{e:defcnstochgrad}
(\forall n\in\NN)\quad c_n=a_n.
\end{equation}
In addition, for every $n\in\NN$, 
$\mathsf{F}=\Fix(\mathsf{J}_{\gamma_n \mathsf{A}}\mathsf{R}_n)$
\cite[Proposition~25.1(iv)]{Livre1} and we deduce from the
firm nonexpansiveness of the operators 
$(\mathsf{J}_{\gamma_n\mathsf{A}})_{n\in\NN}$  
\cite[Corollary~23.8]{Livre1} that
\begin{align}
\label{e:Qr6u5-O-01}
(\forall\mathsf{z}\in\mathsf{F})(\forall n\in\NN)\quad 
\|t_n-\mathsf{z}\|^2+\|r_n-\mathsf{J}_{\gamma_n
\mathsf{A}}r_n-\mathsf{R}_n\mathsf{z}+\mathsf{z}\|^2 
\leq\|r_n-\mathsf{R}_n\mathsf{z}\|^2.
\end{align}
Now set
\begin{equation}
\label{e:deftun}
(\forall n\in\NN)\quad
\widetilde{u}_n=u_n-\EC{u_n}{\XX_n}+\mathsf{B}x_n.
\end{equation}
Then we derive from \eqref{e:Qr6u5-O-01} that 
\eqref{e:condi1} holds with
\begin{equation}
\label{e:fz}
(\forall\mathsf{z}\in\mathsf{F})(\forall n\in\NN)\quad 
\begin{cases}
\theta_{1,n}(\mathsf{z})=
\EC{\|r_n-\mathsf{J}_{\gamma_n \mathsf{A}}r_n- 
\mathsf{R}_n \mathsf{z}+\mathsf{z}\|^2}{\XX_n}\\
\mu_{1,n}(\mathsf{z})=\nu_{1,n}(\mathsf{z})=0\\
s_n(\mathsf{z})=x_n-\gamma_n\widetilde{u}_n-\mathsf{R}_n\mathsf{z}\\
d_n=-\gamma_n (\EC{u_n}{\XX_n}-\mathsf{B}x_n).
\end{cases}
\end{equation}
Thus, \eqref{e:defcnstochgrad}, \eqref{e:fz},
\ref{a:t2i}, \ref{a:t2ii}, and \ref{a:t2iv},  
imply that assumption \ref{a:1ii} in Theorem~\ref{t:1} is satisfied
since 
\begin{align}
\sum_{n\in\NN}\sqrt{\lambda_n\EC{\|d_n\|^2}{\XX_n}}
&\leq 2(\tau_n+1)^{-1}\vartheta\sum_{n\in\NN}
\sqrt{\lambda_n\|\EC{u_n}{\XX_n}-\mathsf{B}x_n\|^2}\nonumber\\
&\leq 2\vartheta
\sum_{n\in\NN}\sqrt{\lambda_n}\|\EC{u_n}{\XX_n}-\mathsf{B}x_n\|
\nonumber\\
&<\pinf.
\end{align}
Moreover, for every $\mathsf{z}\in\mathsf{F}$ and $n\in\NN$, we
derive from \eqref{e:deftun}, 
\eqref{e:2015-02-12a}, and \eqref{e:boundtauetazeta} that
\begin{align}
\label{e:thegrandwazoo}
\EC{\|s_n(\mathsf{z})\|^2}{\XX_n} 
&=\EC{\|x_n-\mathsf{z}-\gamma_n(\widetilde{u}_n-\mathsf{B}\mathsf{z})
\|^2}{\XX_n}\nonumber\\
&=\|x_n-\mathsf{z}\|^2-2\gamma_n
\scal{x_n-\mathsf{z}}{\EC{\widetilde{u}_n}{\XX_n}
-\mathsf{B}\mathsf{z}}+\gamma_n^2
\EC{\|\widetilde{u}_n-\mathsf{B}\mathsf{z}\|^2}{\XX_n}\nonumber\\
&=\|x_n-\mathsf{z}\|^2-2\gamma_n\scal{x_n-\mathsf{z}}{\mathsf{B}
x_n-\mathsf{B}\mathsf{z}}+
\gamma_n^2\big(\EC{\|u_n-\EC{u_n}{\XX_n}\|^2}{\XX_n}\nonumber\\
&\quad\;+2\scal{u_n-\EC{u_n}{\XX_n}}
{\mathsf{B}x_n-\mathsf{B}\mathsf{z}}+\|\mathsf{B}
x_n-\mathsf{B}\mathsf{z}\|^2\big)\nonumber\\
&=\|x_n-\mathsf{z}\|^2-2\gamma_n\scal{x_n-\mathsf{z}}{\mathsf{B}
x_n-\mathsf{B}\mathsf{z}}\nonumber\\
&\quad\;+\gamma_n^2\big(\EC{\|u_n-\EC{u_n}{\XX_n}\|^2}{\XX_n}
+\|\mathsf{B}x_n-\mathsf{B}\mathsf{z}\|^2\big)\nonumber\\
&\leq\|x_n-\mathsf{z}\|^2-\gamma_n(2\vartheta-\gamma_n)
\|\mathsf{B}x_n-\mathsf{B}\mathsf{z}\|^2+\gamma_n^2
\EC{\|u_n-\EC{u_n}{\XX_n}\|^2}{\XX_n}
\nonumber\\
&\leq\|x_n-\mathsf{z}\|^2-
\gamma_n\big(2\vartheta-(1+\tau_n)\gamma_n\big)
\|\mathsf{B}x_n-\mathsf{B}\mathsf{z}\|^2+
\gamma_n^2\zeta_n(\mathsf{z}).
\end{align}
Thus, \eqref{e:condi2} is obtained by setting
\begin{equation}
\label{e:deftheta2stoch}
(\forall n\in\NN)\quad
\begin{cases}
\theta_{2,n}(\mathsf{z})=\gamma_n(2\vartheta-(1+\tau_n)\gamma_n)
\|\mathsf{B} x_n-\mathsf{B}\mathsf{z}\|^2\\
\mu_{2,n}(\mathsf{z})=0\\
\nu_{2,n}(\mathsf{z})=\gamma_n^2 \zeta_n(\mathsf{z}).
\end{cases}
\end{equation}
Altogether, it follows from \ref{a:t2iii} and 
\ref{a:t2iv} that assumption \ref{a:1iii} in 
Theorem~\ref{t:1} is also satisfied. By applying 
Theorem~\ref{t:1}\ref{t:1i}, we deduce from \ref{a:t2iv}, 
\eqref{e:fz}, and \eqref{e:deftheta2stoch} that
\begin{equation}
\label{e:jamar1}
(\forall\mathsf{z}\in\mathsf{F})\quad\sum_{n\in\NN}\lambda_n
\|\mathsf{B}x_n-\mathsf{B}\mathsf{z}\|^2<\pinf
\end{equation}
and 
\begin{equation}
\label{e:2015-01-19c}
(\forall\mathsf{z}\in\mathsf{F})\quad 
\sum_{n\in\NN}\lambda_n\EC{\|r_n-\mathsf{J}_{\gamma_n\mathsf{A}}
r_n-\mathsf{R}_n\mathsf{z}+\mathsf{z}\|^2}{\XX_n}<\pinf.
\end{equation}

\ref{t:2i}: See \eqref{e:jamar1}.

\ref{t:2ii}: Let $\mathsf{z}\in\mathsf{F}$. 
It follows from \eqref{e:es175}, \eqref{e:deftun}, 
\eqref{e:2015-01-22}, and the nonexpansiveness of the operators 
$(\mathsf{J}_{\gamma_n\mathsf{A}})_{n\in\NN}$ that
\begin{align}
(\forall n\in\NN)\quad&
\|x_n-\gamma_n\mathsf{B} x_n-\mathsf{J}_{\gamma_n
\mathsf{A}}(x_n-\gamma_n \mathsf{B} x_n)+\gamma_n \mathsf{B}
\mathsf{z}\|^2\nonumber\\
&\qquad=\|\EC{x_n-\gamma_n \widetilde{u}_n}{\XX_n}-
\mathsf{J}_{\gamma_n\mathsf{A}}(x_n-\gamma_n\mathsf{B}x_n)
+\gamma_n\mathsf{B}\mathsf{z}\|^2\nonumber\\
&\qquad\leq 3\big(\|\EC{r_n-\mathsf{J}_{\gamma_n
\mathsf{A}}r_n+\gamma_n \mathsf{B} \mathsf{z}}{\XX_n}\|^2
+\gamma_n^2\|\EC{u_n}{\XX_n}-\mathsf{B}x_n\|^2\nonumber\\
&\qquad\quad\;+\|\EC{\mathsf{J}_{\gamma_n \mathsf{A}}r_n}{\XX_n}
-\mathsf{J}_{\gamma_n\mathsf{A}}
(x_n-\gamma_n\mathsf{B}x_n)\|^2\big)\nonumber\\
&\qquad\leq 3\big(\EC{\|r_n-\mathsf{J}_{\gamma_n
\mathsf{A}}r_n+\gamma_n\mathsf{B} \mathsf{z}\|^2}{\XX_n}
+\gamma_n^2\EC{\|u_n-\mathsf{B}x_n\|^2}{\XX_n}
\nonumber\\
&\qquad\quad\;+\EC{\|\mathsf{J}_{\gamma_n\mathsf{A}}r_n
-\mathsf{J}_{\gamma_n\mathsf{A}}(x_n-\gamma_n\mathsf{B}x_n)\|^2}
{\XX_n}\big)\nonumber\\
&\qquad\leq 3\big(\EC{\|r_n-\mathsf{J}_{\gamma_n
\mathsf{A}}r_n+\gamma_n \mathsf{B} \mathsf{z}\|^2}{\XX_n}
+\gamma_n^2\EC{\|u_n-\mathsf{B}x_n\|^2}{\XX_n}\nonumber\\
&\qquad\quad\;+\EC{\|r_n-(x_n-\gamma_n\mathsf{B}x_n)\|^2}
{\XX_n}\big)\nonumber\\
&\qquad=3\big(\EC{\|r_n-\mathsf{J}_{\gamma_n\mathsf{A}}
r_n-\mathsf{R}_n\mathsf{z}+\mathsf{z}\|^2}{\XX_n}
+2\gamma_n^2\EC{\|u_n-\mathsf{B}x_n\|^2}{\XX_n}\big)
\nonumber\\
&\qquad\leq3\big(\EC{\|r_n-\mathsf{J}_{\gamma_n\mathsf{A}}
r_n-\mathsf{R}_n\mathsf{z}+\mathsf{z}\|^2}{\XX_n}
+8\vartheta^2\EC{\|u_n-\mathsf{B}x_n\|^2}{\XX_n}\big).
\label{e:astucesstoch}
\end{align}
However, by \eqref{e:boundtauetazeta},
\begin{align}
\label{e:2015-01-19a}
(\forall n\in\NN)\quad
\EC{\|u_n-\mathsf{B}x_n\|^2}{\XX_n}
&\leq2\EC{\|u_n-\EC{u_n}{\XX_n}\|^2+
\|\EC{u_n}{\XX_n}-\mathsf{B}x_n\|^2}{\XX_n}\nonumber\\
&\leq2\big(\tau_n\|\mathsf{B}x_n-\mathsf{B}\mathsf{z}\|^2+
\zeta_n+\|\EC{u_n}{\XX_n}-\mathsf{B}x_n\|^2\big).
\end{align}
Since $\sup_{n\in\NN}\tau_n<\pinf$ by \ref{a:t2iv}, 
we therefore derive from 
\ref{t:2i}, \ref{a:t2ii}, and \ref{a:t2iii} that 
\begin{equation}
\label{e:2015-01-19b}
\sum_{n\in\NN}\lambda_n\EC{\|u_n-\mathsf{B}x_n\|^2}{\XX_n}<\pinf.
\end{equation}
Altogether, the claim follows from \eqref{e:2015-01-19c}, 
\eqref{e:astucesstoch}, and \eqref{e:2015-01-19b}.

\ref{t:2iii}--\ref{t:2iv}: 
Let $\mathsf{z}\in\mathsf{F}$.
We consider the two cases separately.
\begin{itemize}
\item
Suppose that $\inf_{n\in\NN}\lambda_n>0$. We derive from 
\ref{t:2i}, \ref{t:2ii}, and
\ref{a:t2iv} that there exists 
$\widetilde{\Omega}\in\FF$ such that $\PP(\widetilde{\Omega})=1$, 
\begin{equation}
\label{e:focussimA}
(\forall\omega\in\widetilde{\Omega})\quad x_n(\omega)-
\mathsf{J}_{\gamma_n\mathsf{A}}\big(x_n(\omega)-\gamma_n
\mathsf{B} x_n(\omega)\big)\to 0,
\end{equation}
and 
\begin{equation}
\label{e:focussimB}
(\forall\omega\in\widetilde{\Omega})\quad 
\mathsf{B} x_n(\omega)\to\mathsf{B}\mathsf{z}.
\end{equation}
Now set
\begin{equation}
\label{e:burn}
(\forall n\in\NN)\quad y_n
=\mathsf{J}_{\gamma_n\mathsf{A}}(x_n-\gamma_n \mathsf{B}x_n)
\quad\text{and}\quad v_n=\gamma_n^{-1}(x_n-y_n)-\mathsf{B}x_n.
\end{equation}
It follows from \ref{a:t2iv}, \eqref{e:focussimA}, and
\eqref{e:focussimB} that
\begin{equation}
\label{e:2014-12-31d}
(\forall\omega\in\widetilde{\Omega})
\quad
y_n(\omega)-x_n(\omega)\to 0\quad\text{and}\quad v_n(\omega)\to
-\mathsf{B}\mathsf{z}.
\end{equation}
Let $\omega\in\widetilde{\Omega}$. Assume that there exist 
$\mathsf{x}\in\HH$ and a strictly increasing sequence 
$(k_n)_{n\in\NN}$ in $\NN$ such that
$x_{k_n}(\omega)\weakly\mathsf{x}$.
Since $\mathsf{B}x_{k_n}(\omega)\to\mathsf{B}\mathsf{z}$ by
\eqref{e:focussimB} and since $\mathsf{B}$ is maximally monotone 
\cite[Example~20.28]{Livre1}, 
\cite[Proposition~20.33(ii)]{Livre1} yields
$\mathsf{B}\mathsf{x}=\mathsf{B}\mathsf{z}$. 
In addition, \eqref{e:2014-12-31d} implies that
$y_{k_n}(\omega)\weakly\mathsf{x}$ and
$v_{k_n}(\omega)\to-\mathsf{B}\mathsf{z}=-\mathsf{B}\mathsf{x}$. 
Since \eqref{e:burn} entails that 
$(y_{k_n}(\omega),v_{k_n}(\omega))_{n\in\NN}$ 
lies in the graph of $\mathsf{A}$, 
\cite[Proposition~20.33(ii)]{Livre1} asserts that
$-\mathsf{B}\mathsf{x}\in\mathsf{A}\mathsf{x}$, i.e., 
$\mathsf{x}\in\mathsf{F}$. It therefore follows from 
Theorem~\ref{t:1}\ref{t:1iii} that
\begin{equation}
\label{e:2014-12-31a}
x_n(\omega)\weakly x(\omega)
\end{equation}
for every $\omega$ in some $\widehat{\Omega}\in\FF$ such that
$\widehat{\Omega}\subset\widetilde{\Omega}$ and 
$\PP(\widehat{\Omega})=1$. We now turn to the strong convergence
claims. To this end, take $\omega\in\widehat\Omega$.
First, suppose that \ref{a:nyc2014-04-03v} holds. Then
$\mathsf{A}$ is demiregular at 
$x(\omega)$. In view of \eqref{e:2014-12-31d}
and \eqref{e:2014-12-31a}, $y_n(\omega)\weakly x(\omega)$. 
Furthermore, $v_n(\omega)\to-\mathsf{B} x(\omega)$ and 
$(y_n(\omega),v_{n}(\omega))_{n\in\NN}$ 
lies in the graph of $\mathsf{A}$. Altogether
$y_n(\omega)\to x(\omega)$ and
therefore $x_n(\omega)\to x(\omega)$.
Next, suppose that \ref{a:nyc2014-04-03vi} holds. Then,
since \eqref{e:focussimB} yields 
$\mathsf{B}x_n(\omega)\to
\mathsf{B} x(\omega)$, \eqref{e:2014-12-31a} implies that 
$x_n(\omega)\to x(\omega)$.
\item
Suppose that $\sum_{n\in\NN}\tau_n<\pinf$, 
$\sum_{n\in\NN}\lambda_n=\pinf$, 
and $(\forall n\in\NN)$ $\gamma_n=\gamma$.
Let $\mathsf{T}=\mathsf{J}_{\gamma
\mathsf{A}}\circ (\Id -\gamma \mathsf{B})$. 
We deduce from \ref{t:2i} that
\begin{equation}
(\forall\mathsf{z}\in\mathsf{F})\quad
\varliminf\|\mathsf{B}x_n-\mathsf{B}\mathsf{z}\| =0
\end{equation}
and from \ref{t:2ii} that
\begin{equation}
(\forall\mathsf{z}\in\mathsf{F})\quad
\varliminf \|x_n-\mathsf{T} x_n-\gamma (\mathsf{B}
x_n-\mathsf{B}\mathsf{z})\|=0.
\end{equation}
In view of \ref{a:t2iv}, we obtain
\begin{equation}
\label{e:liminfgamconst}
\varliminf \|\mathsf{T} x_n-x_n\|=0.
\end{equation}
In addition, since \ref{a:t2iv} and 
\cite[Proposition~4.33]{Livre1} imply that $\mathsf{T}$ is 
nonexpansive, we derive from \eqref{e:FB} that
\begin{align}
&\hskip -7mm
(\forall n\in\NN)\quad
\|\mathsf{T}x_{n+1}-x_{n+1}\|\nonumber\\
&=\|\mathsf{T}x_{n+1}-(1-\lambda_n) x_n
-\lambda_n (\mathsf{J}_{\gamma
\mathsf{A}}(x_{n}-\gamma u_n)+a_n)\|\nonumber\\
&=\|\mathsf{T} x_{n+1} \!-\!\mathsf{T} x_{n} 
\!-\!(1\!-\!\lambda_n) (x_n\!-\!\mathsf{T}x_{n}) 
-\lambda_n (\mathsf{J}_{\gamma
\mathsf{A}}(x_{n}\!-\!\gamma u_n)
\!-\!\mathsf{J}_{\gamma \mathsf{A}}(x_{n}\!-\!\gamma 
\mathsf{B} x_{n}))\!-\!\lambda_n a_n\|\nonumber\\
&\leq \|\mathsf{T}x_{n+1} -\mathsf{T} x_{n}\|
+(1-\lambda_n)\|\mathsf{T}x_n- x_{n}\|\nonumber\\
&\quad\;+\lambda_n \|\mathsf{J}_{\gamma
\mathsf{A}}(x_{n}-\gamma u_n)
-\mathsf{J}_{\gamma \mathsf{A}}(x_{n}-\gamma \mathsf{B} x_{n})\|
+\lambda_n \|a_n\|\nonumber\\
&\leq\|x_{n+1}-x_n\|
+(1-\lambda_n)\|\mathsf{T}x_n-x_n\|+\lambda_n\gamma
\|u_n-\mathsf{B} x_{n}\|+\lambda_n\|a_n\|\nonumber\\
&=\lambda_n\|\mathsf{J}_{\gamma\mathsf{A}}(x_n-\gamma u_n)+a_n
-x_n\|+(1-\lambda_n)\|\mathsf{T}x_n-x_n\| 
+\lambda_n\gamma \|u_n-\mathsf{B}x_n\|+\lambda_n\|a_n\|\nonumber\\
&\leq\|\mathsf{T}x_n-x_{n}\|+\lambda_n \|\mathsf{J}_{\gamma
\mathsf{A}}(x_{n}-\gamma u_n)-\mathsf{J}_{\gamma
\mathsf{A}}(x_{n}-\gamma \mathsf{B} x_{n})\|
+\lambda_n\gamma \|u_n-\mathsf{B} x_{n}\|
+2\lambda_n\|a_n\|\nonumber\\
&\leq\|\mathsf{T}x_n-x_{n}\|+2\lambda_n
\big(\gamma\|u_n-\mathsf{B}x_{n}\|+\|a_n\|\big).
\label{e:2015-01-19d}
\end{align}
Now set
\begin{equation}
\label{e:2015-01-19e}
(\forall n\in\NN)\quad
\xi_n=\gamma \sqrt{\lambda_n \EC{\|u_n-\mathsf{B}
x_{n}\|^2}{\XX_n}}
+\lambda_n \sqrt{\EC{\|a_n\|^2}{\XX_n}}.
\end{equation}
Using \eqref{e:boundtauetazeta}, we get
\begin{align}
\xi_n
&\leq\gamma \sqrt{\lambda_n
\EC{\|u_n-\EC{u_n}{\XX_n}\|^2}{\XX_n}}
+\gamma\sqrt{\lambda_n\|\EC{u_n}{\XX_n}-\mathsf{B}x_{n}\|^2}
+\lambda_n \sqrt{\EC{\|a_n\|^2}{\XX_n}}\nonumber\\
&\leq\gamma\sqrt{\lambda_n\tau_n} 
\|\mathsf{B}x_{n}-\mathsf{B}\mathsf{z}\|
+\gamma\sqrt{\lambda_n\zeta_n(\mathsf{z})}
+\gamma\sqrt{\lambda_n}\|\EC{u_n}{\XX_n}-\mathsf{B}x_{n}\|
\nonumber\\
&\quad\;+\lambda_n\sqrt{\EC{\|a_n\|^2}{\XX_n}}.
\end{align}
Thus, \eqref{e:2015-01-19d} and \eqref{e:2015-01-21} yield
\begin{align}
\label{e:ineqconstgam}
(\forall n\in\NN)\quad
&\EC{\|\mathsf{T}x_{n+1}-x_{n+1}\|}{\XX_n}\nonumber\\
&\leq\|\mathsf{T}x_{n}-x_n\|
+2\lambda_n\big(\gamma\EC{\|u_n-\mathsf{B}x_n\|}{\XX_n}+
\EC{\|a_n\|}{\XX_n}\big)\nonumber\\
&\leq\|\mathsf{T}x_n-x_n\|+2\xi_n.
\end{align}
In addition, according to the Cauchy-Schwarz inequality and 
\ref{t:2i},
\begin{equation}
\sum_{n\in \NN} \sqrt{\lambda_n \tau_n}
\|\mathsf{B}x_{n}-\mathsf{B}\mathsf{z}\| \leq
\sqrt{\sum_{n\in \NN} \tau_n}
\sqrt{\sum_{n\in \NN} \lambda_n 
\|\mathsf{B}x_{n}-\mathsf{B}\mathsf{z}\|^2}<\pinf.
\end{equation}
Thus, it follows from assumptions 
\ref{a:t2i}-\ref{a:t2iii}
that $(\xi_n)_{n\in\NN}\in\ell_+^1({\mathscr{X}})$, and we deduce
from Proposition~\ref{p:1}\ref{p:1iibis}
and \eqref{e:ineqconstgam}
that $(\|\mathsf{T}x_n- x_n\|)_{n\in\NN}$ converges almost surely.
We then derive from \eqref{e:liminfgamconst} that
there exists $\widetilde{\Omega}\in\FF$ such that
$\PP(\widetilde{\Omega})=1$ and \eqref{e:focussimA} holds.
Let $\omega\in\widetilde{\Omega}$. Suppose that there exist 
$\mathsf{x}\in\HH$ and a strictly increasing sequence 
$(k_n)_{n\in\NN}$ in $\NN$ such that
$x_{k_n}(\omega)\weakly\mathsf{x}$. Since
$x_{k_n}(\omega)\weakly\mathsf{x}$ and 
$\mathsf{T}x_{k_n}(\omega)-x_{k_n}(\omega)\to 0$, 
the demiclosedness principle \cite[Corollary~4.18]{Livre1} asserts 
that $\mathsf{x}\in\mathsf{F}$. Hence, the weak convergence claim
follows from Theorem~\ref{t:1}\ref{t:1iii}. To establish the strong
convergence claims, set $\mathsf{w} =
\mathsf{z}-\gamma \mathsf{B}\mathsf{z}$, and set
$(\forall n\in\NN)$ $w_n=x_n-\gamma\mathsf{B}x_n$. Then 
$\mathsf{T}x_n=\mathsf{J}_{\gamma \mathsf{A}}w_n$ and 
$\mathsf{z}= \mathsf{T}\mathsf{z}=
\mathsf{J}_{\gamma\mathsf{A}}\mathsf{w}$. Hence, appealing to 
the firm nonexpansiveness of $\mathsf{J}_{\gamma\mathsf{A}}$,
we obtain
\begin{align}
\label{e:ehuywt2014-09-07a}
(\forall n\in\NN)\quad
&\scal{\mathsf{T} x_n-\mathsf{z}}{x_n-\mathsf{T}
x_n-\gamma (\mathsf{B} x_n-\mathsf{B}\mathsf{z})}\nonumber\\
&=\scal{\mathsf{T} x_n-\mathsf{z}}{w_n-\mathsf{T}
x_n+\mathsf{z}-\mathsf{w}}\nonumber\\
&=\scal{\mathsf{J}_{\gamma \mathsf{A}}w_n-\mathsf{J}_{\gamma
\mathsf{A}}\mathsf{w}} {(\Id-\mathsf{J}_{\gamma
\mathsf{A}})w_n-(\Id-\mathsf{J}_{\gamma \mathsf{A}})\mathsf{w}}
\nonumber\\
&\geq 0
\end{align}
and therefore
\begin{equation}
\label{e:ehuywt2014-09-07x}
(\forall n\in\NN)\quad
\scal{\mathsf{T}x_n-\mathsf{z}}
{x_n-\mathsf{T}x_n}\geq
\gamma\scal{\mathsf{T}x_n-\mathsf{z}}
{\mathsf{B}x_n-\mathsf{B}\mathsf{z}}.
\end{equation}
Consequently, since $\mathsf{T}$ is nonexpansive and $\mathsf{B}$
satisfies \eqref{e:2015-02-12a},
\begin{align}
\label{e:ehuywt2014-09-07b}
(\forall n\in\NN)\quad
\|x_n-\mathsf{z}\|\,\|\mathsf{T} x_n-x_n\|
&\geq\|\mathsf{T}x_n-\mathsf{z}\|\,\|\mathsf{T} x_n- x_n\|\nonumber\\
&\geq\scal{\mathsf{T}x_n-\mathsf{z}}{x_n-\mathsf{T} x_n}\nonumber\\
&\geq\gamma\scal{\mathsf{T}x_n-\mathsf{z}}
{\mathsf{B}x_n-\mathsf{B}\mathsf{z}}\nonumber\\
&=\gamma\big(\!\scal{\mathsf{T}x_n-x_n}
{\mathsf{B}x_n-\mathsf{B}\mathsf{z}}+
\scal{x_n-\mathsf{z}}{\mathsf{B}x_n-\mathsf{B}\mathsf{z}}\!\big)
\nonumber\\
&\geq-\gamma\|\mathsf{T}x_n-x_n\|\,
\|\mathsf{B}x_n-\mathsf{B}\mathsf{z}\|+
\gamma\vartheta\|\mathsf{B}x_n-\mathsf{B}\mathsf{z}\|^2
\nonumber\\
&\geq-\frac{\gamma}{\vartheta}\|\mathsf{T}x_n-x_n\|\,
\|x_n-\mathsf{z}\|+\gamma\vartheta
\|\mathsf{B}x_n-\mathsf{B}\mathsf{z}\|^2
\end{align}
and hence 
\begin{equation}
(\forall n\in\NN)\quad
\|\mathsf{B}x_n-\mathsf{B}\mathsf{z}\|^2 \leq \frac{1}{\gamma
\vartheta} \Big(1+\frac{\gamma}{\vartheta}\Big)
\|x_n-\mathsf{z}\|\,\|\mathsf{T} x_n-x_n\|.
\end{equation}
Since, $\as$, $(x_n)_{n\in\NN}$ is bounded and 
$\mathsf{T} x_n-x_n\to 0$,
we infer that $\mathsf{B}x_n\to\mathsf{B}\mathsf{z}\;\as$
Thus there exists $\widehat{\Omega}\in\FF$ such that
$\widehat{\Omega}\subset\widetilde{\Omega}$,
$\PP(\widehat{\Omega})=1$, and
\begin{equation}
\label{e:2015-june-17}
(\forall \omega\in\widehat{\Omega})\quad
x_n(\omega)\weakly x(\omega)\quad\text{and}\quad
\mathsf{B} x_n(\omega)\to \mathsf{B}x(\omega).
\end{equation}
Thus, \ref{a:nyc2014-04-03vi} $\Rightarrow$ 
$x_n(\omega)\to x(\omega)$.
Finally, if \ref{a:nyc2014-04-03v} holds, the strong convergence 
of $(x_n(\omega))_{n\in \NN}$ follows from the same arguments
as in the previous case.
\end{itemize}
\end{proof}

\begin{remark} 
\label{r:sunday19}
The demiregularity property in Theorem~\ref{t:2}\ref{t:2iv}
is satisfied by a wide class of operators, e.g., uniformly 
monotone operators or subdifferentials of proper lower 
semicontinuous uniformly convex functions; further examples are 
provided in \cite[Proposition~2.4]{Sico10}.
\end{remark}

\begin{remark} 
\label{r:2015-01-23}
To place our analysis in perspective, we comment on results of 
the literature that seem the most pertinently related 
to Theorem~\ref{t:2}.
\begin{enumerate}
\item
In the deterministic case, Theorem~\ref{t:2}\ref{t:2iii} can be
found in \cite[Corollary~6.5]{Opti04}.
\item
In \cite[Corollary~8]{Atch14}, Problem~\ref{prob:9} is considered
in the special case when $\HH=\RR^N$ and solved via \eqref{e:FB}.
Almost sure convergence properties are established under
the following assumptions: $(\gamma_n)_{n\in\NN}$ is a
decreasing sequence in $\left]0,\vartheta\right]$ such that
$\sum_{n\in\NN}\gamma_n=\pinf$, $\lambda_n\equiv 1$, $a_n\equiv 0$,
and the sequence $(x_n)_{n\in\NN}$ is bounded \emph{a priori}.
\item 
In \cite{Ros14a}, Problem~\ref{prob:1} is addressed using
Algorithm \ref{algo:1}. The authors make the additional
assumptions that 
\begin{equation}
\label{e:unbiased}
(\forall n\in\NN)\quad\EC{u_n}{\XX_n}=\mathsf{B} x_n
\quad\text{and}\quad a_n=0.
\end{equation}
Furthermore they employ vanishing proximal parameters 
$(\gamma_n)_{n\in\NN}$. Almost sure convergence properties of 
the sequence $(x_n)_{n\in\NN}$ are then established under the 
additional assumption that $\mathsf{B}$ is uniformly monotone.
\item
The recently posted paper \cite{Rosa15} employs tools from 
\cite{Siop15} to investigate the convergence of
a variant of \eqref{e:FB} in which no errors 
$(a_n)_{n\in\NN}$ are allowed in the implementation of the 
resolvents, and an inertial term is added, namely,
\begin{multline}
\label{e:ciao}
(\forall n\in\NN) \quad x_{n+1} =
x_n+\lambda_n\big(\mathsf{J}_{\gamma_n\mathsf{A}}
(x_n+\rho_n(x_n-x_{n-1})-\gamma_nu_n)-x_n\big),\\
\text{where}\quad\rho_n\in\left[0,1\right[.
\end{multline}
In the case when $\rho_n\equiv 0$, assertions 
\ref{t:2iii} and \ref{t:2iv}\ref{a:nyc2014-04-03vi} of 
Theorem~\ref{t:2} are obtained under the additional 
hypothesis that $\inf\lambda_n>0$ and the stochastic 
approximations which can be performed are constrained by 
\eqref{e:unbiased}.
\end{enumerate}
\end{remark}

Next, we provide a version of Theorem~\ref{t:1} in which 
a variant of \eqref{e:FB} featuring approximations 
$(\mathsf{A}_n)_{n\in\NN}$ of the operator $\mathsf{A}$ is used.  
In the deterministic forward-backward method, such approximations 
were first used in \cite[Proposition~3.2]{Lema97} (see also 
\cite[Proposition 6.7]{Opti04} for an alternative proof).

\begin{proposition}
\label{p:23}
Consider the setting of Problem~\ref{prob:1}.
Let $x_0$, $(u_n)_{n\in\NN}$, and $(a_n)_{n\in\NN}$ be random 
variables in $L^2(\Omega,\FF,\PP;\HH)$, let 
$(\lambda_n)_{n\in\NN}$ be a sequence in $\rzeroun$, let
$(\gamma_n)_{n\in\NN}$ be a sequence in
$\left]0,2\vartheta\right[$, and let 
$(\mathsf{A}_n)_{n\in\NN}$ be a sequence of maximally monotone
operators from $\HH$ to $2^{\HH}$. Set 
\begin{equation}
\label{e:FBp}
(\forall n\in\NN) \quad x_{n+1} =
x_n+\lambda_n\big(\mathsf{J}_{\gamma_n \mathsf{A}_n}
(x_n-\gamma_nu_n)+a_n-x_n\big).
\end{equation}
Suppose that assumptions \ref{a:t2i-}--\ref{a:t2v} in
Theorem~\ref{t:2} are satisfied, as well as the following: 
\begin{enumerate}[label=\rm(\alph*)]
\setcounter{enumi}{10}
\item 
\label{p:23vi}  
There exist sequences $(\alpha_n)_{n\in \NN}$ and 
$(\beta_n)_{n\in\NN}$ in $\RP$ such that 
$\sum_{n\in\NN}\sqrt{\lambda_n} \alpha_n < \pinf$, 
$\sum_{n\in\NN}\lambda_n\beta_n < \pinf$, and
\begin{equation}
(\forall n\in\NN)(\forall \mathsf{x}\in\HH)\quad
\|\mathsf{J}_{\gamma_n \mathsf{A}_n}\mathsf{x} -
\mathsf{J}_{\gamma_n \mathsf{A}}\mathsf{x}\| 
\leq\alpha_n \|\mathsf{x}\|+\beta_n.
\end{equation}
\end{enumerate}
Then the conclusions of Theorem~\ref{t:2} remain valid.
\end{proposition}
\begin{proof}
Let $\mathsf{z}\in \mathsf{F}$. We have
\begin{equation}
\label{e:lastchance1}
(\forall n\in\NN)\quad
\|x_{n+1}-\mathsf{z}\|\leq (1-\lambda_n)\|x_n-\mathsf{z}\|
+\lambda_n\|\mathsf{J}_{\gamma_n\mathsf{A}_n}
(x_n-\gamma_nu_n)-\mathsf{z}\|+\lambda_n \|a_n\|.
\end{equation}
In addition, 
\begin{align}
\label{e:lastchance2}
(\forall n\in\NN)\quad
&\|\mathsf{J}_{\gamma_n \mathsf{A}_n}
(x_n-\gamma_nu_n)-\mathsf{z}\|\nonumber\\
& \leq\|\mathsf{J}_{\gamma_n \mathsf{A}_n}
(x_n-\gamma_nu_n)-\mathsf{J}_{\gamma_n \mathsf{A}_n}
(\mathsf{z}-\gamma_n \mathsf{B}\mathsf{z})\|
+\|\mathsf{J}_{\gamma_n \mathsf{A}_n} (\mathsf{z}-\gamma_n
\mathsf{B}\mathsf{z})-\mathsf{J}_{\gamma_n \mathsf{A}}
(\mathsf{z}-\gamma_n \mathsf{B}\mathsf{z})\|\nonumber\\
&\leq\|x_n-\gamma_nu_n-\mathsf{z}+\gamma_n \mathsf{B}\mathsf{z}\|
+\|\mathsf{J}_{\gamma_n \mathsf{A}_n} (\mathsf{z}-\gamma_n
\mathsf{B}\mathsf{z})-\mathsf{J}_{\gamma_n \mathsf{A}}
(\mathsf{z}-\gamma_n \mathsf{B}\mathsf{z})\|\nonumber\\
& \leq\|x_n-\mathsf{z}-\gamma_n  (\mathsf{B} x_n -\mathsf{B}
\mathsf{z})-\gamma_n(u_n-\EC{u_n}{\XX_n})\|
+ \gamma_n\|\EC{u_n}{\XX_n}-\mathsf{B} x_n\|\nonumber\\
&\quad\;+\|\mathsf{J}_{\gamma_n \mathsf{A}_n} (\mathsf{z}-\gamma_n
\mathsf{B}\mathsf{z})-\mathsf{J}_{\gamma_n \mathsf{A}}
(\mathsf{z}-\gamma_n \mathsf{B}\mathsf{z})\|.
\end{align}
On the other hand, using assumptions \ref{a:t2iii} and 
\ref{a:t2iv} in Theorem~\ref{t:2} as well as 
\eqref{e:2015-02-12a}, we obtain as in \eqref{e:thegrandwazoo}
\begin{align}
&(\forall n\in\NN)\quad
\EC{\|x_n-\mathsf{z}-\gamma_n  (\mathsf{B} x_n -\mathsf{B}
\mathsf{z}) -\gamma_n (u_n-\EC{u_n}{\XX_n})\|^2}{\XX_n}\nonumber\\
&\hskip 32mm\leq\|x_n-\mathsf{z}\|^2-
\gamma_n\big(2\vartheta-(1+\tau_n)\gamma_n\big)
\|\mathsf{B}x_n-\mathsf{B}\mathsf{z}\|^2+
\gamma_n^2\zeta_n(\mathsf{z})\nonumber\\
&\hskip 32mm\leq\|x_n-\mathsf{z}\|^2 
+\gamma_n^2\zeta_n(\mathsf{z}),
\end{align}
which implies that
\begin{multline}
\label{e:lastchance3}
(\forall n\in\NN)\quad
\EC{\|x_n-\mathsf{z}-\gamma_n(\mathsf{B}x_n-\mathsf{B}
\mathsf{z})-\gamma_n(u_n-\EC{u_n}{\XX_n})\|}{\XX_n}\\
\leq\|x_n-\mathsf{z}\|+\gamma_n \sqrt{\zeta_n(\mathsf{z})}.
\end{multline}
Combining \eqref{e:lastchance1}, \eqref{e:lastchance2}, and
\eqref{e:lastchance3} yields
\begin{align}
\label{e:zoot-allures}
(\forall n\in\NN)\quad
&\EC{\|x_{n+1}-\mathsf{z}\|}{\XX_n}\nonumber\\
&\leq\|x_n-\mathsf{z}\|+\lambda_n\gamma_n\sqrt{\zeta_n(\mathsf{z})}  
+\lambda_n\gamma_n\|\EC{u_n}{\XX_n}-
\mathsf{B} x_n\|
\nonumber\\ 
&\quad\;+\lambda_n\|\mathsf{J}_{\gamma_n \mathsf{A}_n}
(\mathsf{z}-\gamma_n\mathsf{B}\mathsf{z})-\mathsf{J}_{\gamma_n
\mathsf{A}}(\mathsf{z}-\gamma_n \mathsf{B}\mathsf{z})\|
+\lambda_n \EC{\|a_n\|}{\XX_n}\nonumber\\ 
&\leq\|x_n-\mathsf{z}\|+\gamma_n\sqrt{\lambda_n\zeta_n(\mathsf{z})}  
+\gamma_n\sqrt{\lambda_n}\|\EC{u_n}{\XX_n}-\mathsf{B}x_n\|
\nonumber\\ 
&\quad\;+\lambda_n\|\mathsf{J}_{\gamma_n \mathsf{A}_n}
(\mathsf{z}-\gamma_n\mathsf{B}\mathsf{z})-\mathsf{J}_{\gamma_n
\mathsf{A}}(\mathsf{z}-\gamma_n \mathsf{B}\mathsf{z})\|
+\lambda_n\sqrt{\EC{\|a_n\|^2}{\XX_n}}.
\end{align}
Since \cite[Proposition~4.33]{Livre1} asserts that 
\begin{equation}
\label{e:lumpy-gravy}
\text{the operators $(\Id-\gamma_n\mathsf{B})_{n\in\NN}$ are 
nonexpansive,}
\end{equation}
it follows from \ref{p:23vi} that
\begin{align}
(\forall n\in\NN)\quad
\lambda_n\|\mathsf{J}_{\gamma_n \mathsf{A}_n}
(\mathsf{z}-\gamma_n \mathsf{B}\mathsf{z})-\mathsf{J}_{\gamma_n
\mathsf{A}} (\mathsf{z}-\gamma_n \mathsf{B}\mathsf{z})\|
&\leq\sqrt{\lambda_n}\alpha_n\|\mathsf{z}-\gamma_n
\mathsf{B}\mathsf{z}\|+\lambda_n\beta_n\nonumber\\
&\leq\sqrt{\lambda_n}\alpha_n\|\mathsf{z}\|+\lambda_n\beta_n.
\end{align}
Thus,
\begin{equation}
\label{e:laJAdifz}
\sum_{n\in\NN}\lambda_n\|\mathsf{J}_{\gamma_n \mathsf{A}_n}
(\mathsf{z}-\gamma_n \mathsf{B}\mathsf{z})-\mathsf{J}_{\gamma_n
\mathsf{A}}(\mathsf{z}-\gamma_n\mathsf{B}\mathsf{z})\| < \pinf.
\end{equation}
In view of assumptions \ref{a:t2i-}-\ref{a:t2iv} in
Theorem~\ref{t:2} and \eqref{e:laJAdifz}, we deduce from 
\eqref{e:zoot-allures} and Proposition~\ref{p:1}\ref{p:1ii} 
that $(x_n)_{n\in\NN}$ is almost surely bounded. In turn,
\eqref{e:lumpy-gravy} asserts that 
$(x_n-\gamma_n\mathsf{B} x_n)_{n\in\NN}$ is likewise.
Now set
\begin{equation}
(\forall n\in\NN)\quad\widetilde{a}_n=
\mathsf{J}_{\gamma_n \mathsf{A}_n}
(x_n-\gamma_nu_n)-\mathsf{J}_{\gamma_n\mathsf{A}}
(x_n-\gamma_nu_n)+a_n.
\end{equation}
Then \eqref{e:FBp} can be rewritten as
\begin{equation}
\label{e:Huy9i-23}
(\forall n\in\NN) \quad x_{n+1} =
x_n+\lambda_n\big(\mathsf{J}_{\gamma_n \mathsf{A}}
(x_n-\gamma_nu_n)+\widetilde{a}_n-x_n\big).
\end{equation}
However,
\begin{multline}
(\forall n \in \NN)\quad
\sqrt{\EC{\|\widetilde{a}_n\|^2}{\XX_n}}
\leq\sqrt{\EC{\|\mathsf{J}_{\gamma_n
\mathsf{A}_n}(x_n-\gamma_nu_n)-\mathsf{J}_{\gamma_n
\mathsf{A}}(x_n-\gamma_nu_n)\|^2}{\XX_n}}\\
+\sqrt{\EC{\|a_n\|^2}{\XX_n}}.
\end{multline}
On the other hand, according to \ref{p:23vi}, 
assumption \ref{a:t2iii} in Theorem~\ref{t:2}, and
\eqref{e:lumpy-gravy},
\begin{align}
(\forall n\in\NN)\quad
&\lambda_n\sqrt{\EC{\|\mathsf{J}_{\gamma_n
\mathsf{A}_n}(x_n-\gamma_nu_n)-\mathsf{J}_{\gamma_n
\mathsf{A}}(x_n-\gamma_nu_n)\|^2}{\XX_n}} \nonumber\\
&\leq\lambda_n \sqrt{\EC{(\alpha_n
\|x_n-\gamma_nu_n\|+\beta_n)^2}{\XX_n}}\nonumber\\
&\leq\lambda_n \sqrt{\EC{(\alpha_n\|x_n-\gamma_n\mathsf{B}x_n\|
+\gamma_n\|u_n-\mathsf{B}x_n\|+\beta_n)^2}{\XX_n}}
\nonumber\\
&\leq\lambda_n\alpha_n\big(\|x_n-\gamma_n\mathsf{B}x_n\|+
\gamma_n\sqrt{\EC{\|u_n-\mathsf{B}
x_n\|^2}{\XX_n}}\big)+\lambda_n\beta_n\nonumber\\
&\leq\lambda_n\alpha_n\big(\|x_n-\gamma_n\mathsf{B}x_n\|+
\gamma_n\|\EC{u_n}{\XX_n}-\mathsf{B} x_n\|\nonumber\\
&\quad\;+\gamma_n\sqrt{\EC{\|u_n-\EC{u_n}{\XX_n}\|^2}{\XX_n}}\big)
+\lambda_n\beta_n\nonumber\\
&\leq\lambda_n\alpha_n\big(\|x_n-\gamma_n\mathsf{B}x_n\|+
\gamma_n\|\EC{u_n}{\XX_n}-\mathsf{B} x_n\|+\gamma_n 
\sqrt{\tau_n}\|\mathsf{B} x_n-\mathsf{B} \mathsf{z}\|\nonumber\\
&\quad\;+\gamma_n\sqrt{\zeta_n(\mathsf{z})}\big)+
\lambda_n\beta_n \nonumber\\
&\leq\sqrt{\lambda_n}\alpha_n\big(\|x_n-\gamma_n\mathsf{B}x_n\|+
\gamma_n\sqrt{\lambda_n}\|\EC{u_n}{\XX_n}-\mathsf{B} x_n\|
+\gamma_n\sqrt{\tau_n}\|\mathsf{B}x_n-\mathsf{B}\mathsf{z}\|
\nonumber\\
&\quad\;+\gamma_n
\sqrt{\lambda_n\zeta_n(\mathsf{z})}\big)+\lambda_n\beta_n.
\end{align}
However, assumptions \ref{a:t2ii} and \ref{a:t2iii} in
Theorem~\ref{t:2} guarantee that 
$(\sqrt{\lambda_n}\|\EC{u_n}{\XX_n}-\mathsf{B}x_n\|)_{n\in\NN}$ 
and $(\sqrt{\lambda_n\zeta_n(\mathsf{z})})_{n\in\NN}$ are 
$\as$ bounded. Since $(\mathsf{B} x_n)_{n\in\NN}$ and 
$(x_n-\gamma_n\mathsf{B} x_n)_{n\in\NN}$ are likewise,
it follows from \ref{p:23vi} and \eqref{e:lumpy-gravy} that
\begin{equation}
\sum_{n\in\NN}
\lambda_n\sqrt{\EC{\|\mathsf{J}_{\gamma_n
\mathsf{A}_n}(x_n-\gamma_nu_n)-\mathsf{J}_{\gamma_n
\mathsf{A}}(x_n-\gamma_nu_n)\|^2}{\XX_n}}<\pinf,
\end{equation}
and consequently that
\begin{equation}
\sum_{n\in\NN}\lambda_n\sqrt{\EC{\|\widetilde{a}_n\|^2}{\XX_n}}
<\pinf.
\end{equation}
Applying Theorem~\ref{t:2} to algorithm \eqref{e:Huy9i-23} then
yields the claims.
\end{proof}

\section{Applications} 
\label{sec:5}

As discussed in the Introduction, the forward-backward algorithm is
quite versatile and it can be applied in various forms. Many
standard applications of Theorem~\ref{t:2} can of course be recovered
for specific choices of $\mathsf{A}$ and $\mathsf{B}$, in
particular Problem~\ref{prob:9}. 
Using the product space framework of
\cite{Sico10}, it can also be applied to solve systems of
coupled monotone inclusions. On the other hand, using the approach
proposed in \cite{Svva10,Opti14}, it can be used to solve 
strongly monotone composite inclusions (in particular, strongly
convex composite minimization problems), say, 
\begin{equation}
\label{e:fPrimal}
\text{find}\;\;\mathsf{x}\in\HH\;\;\text{such that}\;\;
\mathsf{z}\in\mathsf{A}\mathsf{x}+\sum_{k=1}^q\mathsf{L}_k^*
\big((\mathsf{B}_k\infconv\mathsf{D}_k)
(\mathsf{L}_k\mathsf{x}-\mathsf{r}_k)\big)+\rho\mathsf{x},
\end{equation}
since their dual problems
assume the general form of Problem~\ref{prob:1} and the primal
solution can trivially be recovered from any dual solution.
In \eqref{e:fPrimal}, $\mathsf{z}\in\HH$, $\rho\in\RPP$ and, 
for every $k\in\{1,\ldots,q\}$, 
$\mathsf{r}_k$ lies in a real Hilbert space $\GG_k$, 
$\mathsf{B}_k\colon\GG_k\to 2^{\GG_k}$ is maximally monotone, 
$\mathsf{D}_k\colon\GG_k\to 2^{\GG_k}$ is maximally monotone and
strongly monotone, 
$\mathsf{B}_k\infconv\mathsf{D}_k=(\mathsf{B}_k^{-1}+
\mathsf{D}_k^{-1})^{-1}$, and $\mathsf{L}_k\in\BL(\HH,\GG_k)$. In
such instances the forward-backward algorithm actually yields a
primal-dual method which produces a sequence converging to
the primal solution (see \cite[Section~5]{Opti14} for details).
Now suppose that, in addition, $\mathsf{C}\colon\HH\to\HH$ is
cocoercive. As in \cite{Svva12}, consider the primal problem
\begin{equation}
\label{e:fprimal}
\text{find}\;\;\mathsf{x}\in\HH\;\;\text{such that}\;\;
\mathsf{z}\in\mathsf{A}\mathsf{x}+\sum_{k=1}^q\mathsf{L}_k^*
\big((\mathsf{B}_k\infconv\mathsf{D}_k)
(\mathsf{L}_k\mathsf{x}-\mathsf{r}_k)\big)+\mathsf{C}\mathsf{x},
\end{equation}
together with the dual problem 
\begin{multline}
\label{e:fdual}
\text{find}\;\;\mathsf{v}_1\in\GG_1,\:\ldots,\:
\mathsf{v}_q\in\GG_q\quad\text{such that}\\
(\forall k\in\{1,\ldots,q\})\quad
-\mathsf{r}_k\in-\mathsf{L}_k^*(\mathsf{A}+\mathsf{C})^{-1}
\bigg(\mathsf{z}-\sum_{l=1}^q\mathsf{L}_{l}^*\mathsf{v}_l\bigg)
+\mathsf{B}_k^{-1}\mathsf{v}_k+\mathsf{D}_k^{-1}\mathsf{v}_k.
\end{multline}
Using renorming techniques in the primal-dual space 
going back to \cite{Heyu12} in the context of finite-dimensional
minimization problems, the primal-dual problem
\eqref{e:fprimal}--\eqref{e:fdual} can be reduced to an instance 
of Problem~\ref{prob:1} \cite{Opti14,Bang13} (see also
\cite{Cond13}) and therefore solved via Theorem~\ref{t:2}. 
Next, we explicitly illustrate an 
application of this approach in the special case when
\eqref{e:fprimal}--\eqref{e:fdual} is a minimization
problem.

\subsection{A stochastic primal-dual minimization method} 
We denote by $\Gamma_0(\HH)$ the class of proper lower 
semicontinuous convex functions. The Moreau subdifferential of 
$\mathsf{f}\in\Gamma_0(\HH)$ is the maximally monotone operator
\begin{equation}
\label{e:subdiff}
\partial\mathsf{f}\colon\HH\to 2^{\HH}\colon\mathsf{x}\mapsto
\menge{\mathsf{u}\in\HH}{(\forall\mathsf{y}\in\HH)\;\;
\scal{\mathsf{y}-\mathsf{x}}{\mathsf{u}}+\mathsf{f}(\mathsf{x})
\leq\mathsf{f}(\mathsf{y})}.
\end{equation}
The inf-convolution of $\mathsf{f}\colon\HH\to\RX$
and $\mathsf{h}\colon\HH\to\RX$ is defined as 
$\mathsf{f}\infconv \mathsf{h}\colon\HH\to\RXX\colon
\mathsf{x}\mapsto\inf_{\mathsf{y}\in\HH}\big(\mathsf{f}(\mathsf{y})
+\mathsf{h}(\mathsf{x}-\mathsf{y})\big)$.
The conjugate of a function $\mathsf{f}\in \Gamma_0(\HH)$ is
the function $\mathsf{f}^*\in\Gamma_0(\HH)$ defined by
$(\forall\mathsf{u}\in\HH)$ $\mathsf{f}^*(\mathsf{u})=
\sup_{\mathsf{x}\in\HH}(\scal{\mathsf{x}}{\mathsf{u}}-
\mathsf{f}(\mathsf{x}))$.
Let $\mathsf{U}$ be a strongly positive self-adjoint operator in
$\BL(\HH)$. The proximity operator of 
$\mathsf{f}\in\Gamma_0(\HH)$ relative to 
the metric induced by $\mathsf{U}$ is 
\begin{equation}
\prox^{\mathsf{U}}_ {\mathsf{f}}\colon\HH\to\HH\colon \mathsf{x}\to
\argmind{\mathsf{y}\in\HH}{\mathsf{f}(\mathsf{y})+\frac12
\|\mathsf{x}-\mathsf{y}\|^2_{\mathsf{U}}},
\end{equation}
where
\begin{equation}
(\forall\mathsf{x}\in\HH)\qquad 
\|\mathsf{x}\|_{\mathsf{U}}=
\sqrt{\scal{\mathsf{x}}{\mathsf{U}\mathsf{x}}}.
\end{equation}
We have $\prox^{\mathsf{U}}_{\mathsf{f}} =
\mathsf{J}_{\mathsf{U}^{-1}\partial\mathsf{f}}$.

We apply Theorem~\ref{t:2} to derive a stochastic version of 
a primal-dual optimization algorithm for solving a multivariate
optimization problem which was first proposed in 
\cite[Section~4]{Svva12}.

\begin{problem}
\label{prob:3}
Let $\mathsf{f}\in \Gamma_0(\HH)$, let $\mathsf{h}\colon\HH\to\RR$ 
be convex and differentiable with a Lipschitz-continuous gradient,
and let $q$ be a strictly positive integer. For every 
$k\in\{1,\ldots,q\}$, let $\GG_k$ be a separable 
Hilbert space, let 
$\mathsf{g}_k\in\Gamma_0(\GG_k)$, 
let $\mathsf{j}_k\in\Gamma_0(\GG_k)$ be strongly convex,
and let $\mathsf{L}_{k}\in \BL(\HH,\GG_k)$. 
Let $\boldsymbol{\GG}=\GG_1\oplus\cdots\oplus\GG_q$ be the
direct Hilbert sum of $\GG_1,\ldots,\GG_q$, and suppose that
there exists $\overline{\mathsf{x}}\in\HH$ such that 
\begin{equation}\label{e:qualif}
0\in\partial\mathsf{f}(\overline{\mathsf{x}})
+\sum_{k=1}^q\mathsf{L}_{k}^*(\partial\mathsf{g}_k\infconv
\partial\mathsf{j}_k)(\mathsf{L}_{k}\overline{\mathsf{x}})
+\nabla\mathsf{h}(\overline{\mathsf{x}}).
\end{equation}
Let $\mathsf{F}$ be the set of solutions to the problem
\begin{equation}
\label{e:primopt}
\minimize{\mathsf{x}\in\HH}
{\mathsf{f}(\mathsf{x})+
\sum_{k=1}^q (\mathsf{g}_k\infconv\mathsf{j}_k)
(\mathsf{L}_{k}\mathsf{x})}+\mathsf{h}(\mathsf{x})
\end{equation}
and let $\boldsymbol{\mathsf{F}}^*$
be the set of  solutions to the dual problem
\begin{equation}
\minimize{\boldsymbol{\mathsf{v}}\in\boldsymbol{\GG}}
{(\mathsf{f}^*\infconv \mathsf{h}^*) \bigg(-\Sum_{k=1}^q
\mathsf{L}_{k}^*\mathsf{v}_{k}\bigg)+\sum_{k=1}^q
\big(\mathsf{g}_k^*(\mathsf{v}_k)+\mathsf{j}_k^*
(\mathsf{v}_k)\big)},
\end{equation}
where we denote by $\boldsymbol{\mathsf{v}}=
(\mathsf{v}_1,\ldots,\mathsf{v}_q)$ a generic 
point in $\boldsymbol{\GG}$.
The problem is to find a point in 
$\mathsf{F}\times\boldsymbol{\mathsf{F}}^*$.
\end{problem}

We address the case when only stochastic approximations of the 
gradients of $\mathsf{h}$ and 
$(\mathsf{j}_k^*)_{1\leq k \leq q}$ and approximations 
of the functions $\mathsf{f}$ are available to solve
Problem~\ref{prob:3}.

\begin{algorithm}
\label{algo:7}
Consider the setting of Problem~\ref{prob:3} and let
$\mathsf{W}\in\BL(\HH)$ be strongly positive and self-adjoint. 
Let $(\mathsf{f}_n)_{n\in\NN}$ be a sequence in $\Gamma_0(\HH)$,
let $(\lambda_n)_{n\in\NN}$ be a sequence in $\left]0,1\right]$ 
such that $\sum_{n\in\NN}\lambda_n=\pinf$, and, for every 
$k\in\{1,\ldots,q\}$, let $\mathsf{U}_k\in\BL(\GG_k)$ be 
strongly positive and self-adjoint.
Let $x_0$, $(u_n)_{n\in\NN}$, and
$(b_n)_{n\in\NN}$ be random variables in $L^2(\Omega,\FF,\PP;\HH)$,
and let $\boldsymbol{v}_0$, $(\boldsymbol{s}_n)_{n\in\NN}$,
and $(\boldsymbol{c}_n)_{n\in\NN}$ be random variables in 
$L^2(\Omega,\FF,\PP;\GGG)$.
Iterate
\begin{equation}
\label{e:PDcoordopt1}
\begin{array}{l}
\text{for}\;n=0,1,\ldots\\
\left\lfloor
\begin{array}{l}
\displaystyle y_{n} =
\prox_{\mathsf{f}_n}^{\mathsf{W}^{-1}}
\bigg(x_{n}-\mathsf{W}\bigg(\sum_{k=1}^q {\mathsf{L}^*_{k} v_{k,n}}
+u_n\bigg)\bigg)+b_{n}\\
x_{n+1}=x_{n}+\lambda_n (y_{n}-x_{n})\\
\text{for}\;k=1,\ldots,q\\
\left\lfloor
\begin{array}{l}
\displaystyle
w_{k,n}=\prox_{\mathsf{g}_{k}^*}^{\mathsf{U}_k^{-1}}
\big(v_{k,n}+\mathsf{U}_k(\mathsf{L}_{k}
(2y_{n}-x_{n})- s_{k,n})\big)+c_{k,n}\\
v_{k,n+1}=v_{k,n}+\lambda_n (w_{k,n}-v_{k,n}).
\end{array}
\right.
\end{array}
\right.\\
\end{array}
\end{equation}
\end{algorithm}

\begin{proposition}
\label{p:3}
Consider the setting of Problem~\ref{prob:3}, let 
$\mathscr{X}=(\XX_n)_{n\in\NN}$ be a sequence of sub-sigma-algebras 
of $\FF$, and let $(x_n)_{n\in\NN}$ and 
$(\boldsymbol{v}_n)_{n\in \NN}$ be sequences generated by 
Algorithm~\ref{algo:7}. Let $\mu\in\RPP$ be a Lipschitz constant 
of the gradient of $\mathsf{h} \circ \mathsf{W}^{1/2}$ and, 
for every $k\in \{1,\ldots,q\}$, let $\nu_k\in\RPP$ be a Lipschitz 
constant of the gradient of $\mathsf{j}_k^*\circ\mathsf{U}_k^{1/2}$.
Assume that the following hold:
\begin{enumerate}[label=\rm(\alph*)]
\item
$(\forall n\!\in\!\NN)$ 
$\sigma(x_{n'},\boldsymbol{v}_{n'})_{0\leq n'\leq n}\subset
\XXX_n\subset\boldsymbol{\XX_{n+1}}$.
\item 
\label{a:p3i}
$\sum_{n\in\NN}\lambda_n\sqrt{\EC{\| b_n\|^2}
{\boldsymbol{\XX}_n}}<\pinf$ and
$\sum_{n\in\NN}\lambda_n\sqrt{\EC{\| \boldsymbol{c}_n\|^2}
{\boldsymbol{\XX}_n}}<\pinf$.
\item 
\label{a:p3ii}
$\sum_{n\in\NN}\sqrt{\lambda_n}\|\EC{u_n}{\boldsymbol{\XX}_n}-
\nabla\mathsf{h}(x_{n}) \|<\pinf$.
\item 
\label{a:p3iii} 
For every $k\in \{1,\ldots,q\}$,
$\sum_{n\in\NN}\sqrt{\lambda_n}
\|\EC{s_{k,n}}{\boldsymbol{\XX}_n}-\nabla
\mathsf{j}_k^*(v_{k,n})\|<\pinf$.
\item 
\label{a:p3iv}
There exists a summable sequence $(\tau_n)_{n\in\NN}$ in $\RP$ such 
that, for every $(\mathsf{x},\boldsymbol{\mathsf{v}})\in \mathsf{F} 
\times \boldsymbol{\mathsf{F}}^*$, there exists 
$\big(\zeta_n(\mathsf{x},\boldsymbol{\mathsf{v}})\big)_{n\in\NN}\in
\ell^\infty_+({\mathscr{X}})$ such that
$\big(\lambda_n\zeta_n(\mathsf{x},
\boldsymbol{\mathsf{v}})\big)_{n\in\NN}\in
\ell_+^{1/2}({\mathscr{X}})$ and
\begin{multline}
\label{e:boundtauetazetaPD}
(\forall n\in\NN)\quad 
\EC{\|u_n-\EC{u_n}{\boldsymbol{\XX}_n}\|^2}{\boldsymbol{\XX}_n}+
\EC{\|\boldsymbol{s}_n-\EC{\boldsymbol{s}_n}{\boldsymbol{\XX}_n}\|^2}
{\boldsymbol{\XX}_n}\\
\leq\tau_n\bigg(\|\nabla\mathsf{h}(x_n)-\nabla\mathsf{h}
(\mathsf{x})\|^2
+\sum_{k=1}^q\|\nabla\mathsf{j}_k^*(v_{k,n})-
\nabla\mathsf{j}_k^*(\mathsf{v}_k)\|^2\bigg)
+\zeta_n(\mathsf{x},\boldsymbol{\mathsf{v}}).
\end{multline}
\item 
\label{a:p3v}
There exist sequences $(\alpha_{n})_{n\in \NN}$ and 
$(\beta_{n})_{n\in \NN}$ in $\RP$ such that 
$\sum_{n\in \NN} \sqrt{\lambda_n} \alpha_{n} < \pinf$, 
$\sum_{n\in\NN} \lambda_n \beta_{n} < \pinf$, and
\begin{equation}
(\forall n\in\NN)(\forall \mathsf{x}\in\HH)\quad
\|\prox_{\mathsf{f}_n}^{\mathsf{W}^{-1}}\mathsf{x} -
\prox_{\mathsf{f}}^{\mathsf{W}^{-1}}\mathsf{x}\| 
\leq\alpha_{n} \|\mathsf{x}\|+\beta_{n}.
\end{equation}
\item  
\label{a:p3vi}
$\max\{\mu,\nu_1,\ldots,\nu_q\}< 
2\left(1-\sqrt{\sum_{k=1}^q \|\mathsf{U}_k^{1/2}
\mathsf{L}_{k}\mathsf{W}^{1/2} \|^2}\right)$.
\end{enumerate}
Then, the following hold for some $\mathsf{F}$-valued random
variable $x$ and some $\boldsymbol{\mathsf{F}}^*$-valued random 
variable $\boldsymbol{v}$:
\begin{enumerate}
\item  
$(x_n)_{n\in\NN}$ converges weakly $\as$ to $x$ and
$(\boldsymbol{v}_n)_{n\in\NN}$ converges weakly almost surely 
to $\boldsymbol{v}$. 
\item  
Suppose that $\nabla\mathsf{h}$ is 
demiregular at every $\mathsf{x}\in\mathsf{F}$.
Then $(x_n)_{n\in\NN}$ converges strongly almost surely to $x$. 
\item
Suppose that there exists $k\in \{1,\ldots,q\}$ such that,
for every $\boldsymbol{\mathsf{v}}\in\boldsymbol{\mathsf{F}}^*$,
$\nabla\mathsf{j}_k^{*}$ is demiregular at $\mathsf{v}_k$.
Then $(v_{k,n})_{n\in\NN}$ converges strongly almost surely to 
$v_k$.
\end{enumerate}
\end{proposition}
\begin{proof}
The proof relies on the ability to employ a constant proximal 
parameter in algorithm \eqref{e:FBp}.
Let us define $\KKK=\HH\oplus\GGG$, $\boldsymbol{\mathsf g}
\colon\GGG\to\RX\colon
\boldsymbol{\mathsf{v}}\mapsto\sum_{k=1}^q 
\mathsf{g}_k(\mathsf{v}_k)$, 
$\boldsymbol{\mathsf j}\colon\GGG\to\RX\colon
\boldsymbol{\mathsf{v}}\mapsto
\sum_{k=1}^q \mathsf{j}_k(\mathsf{v}_k)$, 
$\boldsymbol{\mathsf{L}}\colon\HH\to\GGG\colon
{\mathsf{x}}\mapsto\big(\mathsf{L}_{k}
\mathsf{x}\big)_{1\leq k\leq q}$,
and $\boldsymbol{\mathsf{U}}\colon\GGG\to\GGG\colon
\boldsymbol{\mathsf{v}}\mapsto (\mathsf{U}_1
\mathsf{v}_1,\ldots,\mathsf{U}_q\mathsf{v}_q)$.
Let us now introduce the set-valued operator
\begin{equation}
\label{e:maximal1}
\boldsymbol{\mathsf{A}}\colon\KKK\to 2^{\KKK}\colon
({\mathsf{x}},\boldsymbol{\mathsf{v}})\mapsto 
\big(\partial{\mathsf{f}}({\mathsf{x}})+\boldsymbol{\mathsf{L}}^*
\boldsymbol{\mathsf{v}}\big)\times\big(-\boldsymbol{\mathsf{L}}
{\mathsf{x}}+\partial\boldsymbol{\mathsf{g}}^{*}
(\boldsymbol{\mathsf{v}})\big),
\end{equation}
the single-valued operator
\begin{equation}
\label{e:maximal21}
\boldsymbol{\mathsf B}\colon\KKK\to\KKK\colon
({\mathsf{x}},\boldsymbol{\mathsf{v}})\mapsto 
\big(\nabla {\mathsf h}({\mathsf{x}}),
\nabla \boldsymbol{\mathsf j}^{*}(\boldsymbol{\mathsf{v}})\big),
\end{equation}
and the bounded linear operator
\begin{align} 
\label{e:invV1}
\boldsymbol{\mathsf V}\colon\KKK\to\KKK\colon
({\mathsf{x}},\boldsymbol{\mathsf{v}})\mapsto 
\big({\mathsf W}^{-1}
{\mathsf{x}}-\boldsymbol{\mathsf{L}}^*\boldsymbol{\mathsf{v}},
-\boldsymbol{\mathsf{L}}{\mathsf{x}}+\boldsymbol{\mathsf U}^{-1}
\boldsymbol{\mathsf{v}}\big). 
\end{align}
Further, set
\begin{equation}
\label{e:nonnon}
\vartheta=\left(1-\sqrt{\sum_{k=1}^q 
\|\mathsf{U}_k^{1/2} \mathsf{L}_{k}
\mathsf{W}^{1/2} \|^2}\:\right)\min\{\mu^{-1},\nu_1^{-1},
\ldots,\nu_q^{-1}\}
\end{equation}
and 
\begin{equation}
\label{e:X}
(\forall n\in\NN)\quad 
\widetilde{\tau}_n=
\|\boldsymbol{\mathsf{V}}^{-1}\|\,
\|\boldsymbol{\mathsf{V}}\|\tau_n.
\end{equation}
Since \ref{a:p3iv} imposes that 
$\sum_{n\in\NN}\widetilde{\tau_n}<\pinf$, we 
assume without loss of generality that
\begin{equation}
\label{e:CentraleP}
\sup_{n\in\NN}\,\widetilde{\tau_n}<2\vartheta-1.
\end{equation}
In the renormed space $(\KKK,\|\cdot\|_{\boldsymbol{\mathsf{V}}})$,
$\boldsymbol{\mathsf{V}}^{-1}\boldsymbol{\mathsf{A}}$ is
maximally monotone and 
$\boldsymbol{\mathsf{V}}^{-1}\boldsymbol{\mathsf B}$
is cocoercive \cite[Lemma~3.7]{Opti14} with 
cocoercivity constant $\vartheta$ \cite[Lemma~4.3]{Repe15}.
In addition, finding a zero of the sum of these
operators is equivalent to finding a point in $\mathsf{F}\times
\boldsymbol{\mathsf{F}}^*$, and algorithm \eqref{e:FBp}
with $\gamma_n \equiv 1$ for solving this monotone inclusion
problem specializes to \eqref{e:PDcoordopt1} 
(see \cite{Opti14,Repe15} for details), which can thus be 
rewritten as
\begin{equation}
\label{e:FBp2}
(\forall n\in\NN) \quad (x_{n+1},\boldsymbol{v}_{n+1}) =
(x_n,\boldsymbol{v}_n)+
\lambda_n\big(\mathsf{J}_{\boldsymbol{\mathsf{V}}^{-1}
\boldsymbol{\mathsf{A}}_n}
\big((x_n,\boldsymbol{v}_n) -
\boldsymbol{\mathsf{V}}^{-1}(u_n,\boldsymbol{s}_n)\big)
+\boldsymbol{a}_n-(x_n,\boldsymbol{v}_n)\big),
\end{equation}
where
\begin{equation}
(\forall n\in\NN)\quad
\boldsymbol{\mathsf{a}}_n=(b_n,\boldsymbol{c}_n)
\end{equation}
and
\begin{equation}
(\forall n\in\NN)\quad
\boldsymbol{\mathsf{A}}_n\colon\KKK\to 2^{\KKK}
\colon({\mathsf{x}},\boldsymbol{\mathsf{v}})\mapsto
\big(\partial{\mathsf{f}_n}({\mathsf{x}})+\boldsymbol{\mathsf{L}}^*
\boldsymbol{\mathsf{v}}\big) \times\big(-\boldsymbol{\mathsf{L}}
{\mathsf{x}}+\partial\boldsymbol{\mathsf{g}}^{*}
(\boldsymbol{\mathsf{v}})\big).
\end{equation}
Then 
\begin{multline}
\label{e:beforemoulinex}
(\forall n\in\NN)
(\forall ({\mathsf{x}},\boldsymbol{\mathsf{v}})\in \KKK)\quad
\mathsf{J}_{\boldsymbol{\mathsf{V}}^{-1}
\boldsymbol{\mathsf{A}}_n}({\mathsf{x}},
\boldsymbol{\mathsf{v}}) =
\Big(\mathsf{y},
\prox_{\boldsymbol{\mathsf{g}}^{*}}^{\boldsymbol{\mathsf{U}}^{-1}}
\big(\boldsymbol{\mathsf{v}}+\boldsymbol{\mathsf{U}}
\boldsymbol{\mathsf{L}}(2\mathsf{y}-\mathsf{x})\big)\Big),\\
\text{where}\quad
\mathsf{y} = \prox_{\mathsf{f}_n}^{\mathsf{W}^{-1}}(\mathsf{x}-
\mathsf{W}\boldsymbol{\mathsf{L}}^*\boldsymbol{\mathsf{v}}).
\end{multline}
Assumption \ref{a:p3i} is equivalent to $\sum_{n\in\NN}\lambda_n
\sqrt{\EC{\|\boldsymbol{a}_n\|_{\boldsymbol{\mathsf{V}}}^2}
{\boldsymbol{\XX}_n}}<\pinf$, and
assumptions \ref{a:p3ii} and \ref{a:p3iii} imply that
\begin{equation}
\sum_{n\in\NN}\sqrt{\lambda_n}\|\EC{\boldsymbol{\mathsf{V}}^{-1}
(u_n,\boldsymbol{s}_n)}{\boldsymbol{\XX}_n}-
\boldsymbol{\mathsf{V}}^{-1}\boldsymbol{\mathsf B}
(u_n,\boldsymbol{s}_n)\|_{\boldsymbol{\mathsf{V}}}<\pinf.
\end{equation}
For every $(\mathsf{x},\boldsymbol{\mathsf{v}})\in \mathsf{F} 
\times \boldsymbol{\mathsf{F}}^*$, assumption \ref{a:p3iv} yields
\begin{align}
(\forall n\in\NN)\quad 
&\EC{\|\boldsymbol{\mathsf{V}}^{-1}(u_n,\boldsymbol{s}_n)
-\EC{\boldsymbol{\mathsf{V}}^{-1}(u_n,\boldsymbol{s}_n)}
{\boldsymbol{\XX}_n}\|_{\boldsymbol{\mathsf{V}}}^2}
{\boldsymbol{\XX}_n}\nonumber\\
&\le \|\boldsymbol{\mathsf{V}}^{-1}\| \big(
\EC{\|u_n-\EC{u_n}{\boldsymbol{\XX}_n}\|^2}{\boldsymbol{\XX}_n}
+ \EC{\|\boldsymbol{s}_n
-\EC{\boldsymbol{s}_n}{\boldsymbol{\XX}_n}\|^2}
{\boldsymbol{\XX}_n}\big)\nonumber\\
&\leq\|\boldsymbol{\mathsf{V}}^{-1}\|\big(
\tau_n\big(\|\nabla \mathsf{h}(x_n)-\nabla\mathsf{h}(\mathsf{x})\|^2
+\|\nabla \boldsymbol{\mathsf{j}}^*(\boldsymbol{v}_n)-\nabla
\boldsymbol{\mathsf{j}}^*(\boldsymbol{\mathsf{v}})\|^2\big)
+\zeta_n(\mathsf{x},\boldsymbol{\mathsf{v}})\big)\nonumber\\
&\leq\widetilde{\tau}_n\|\boldsymbol{\mathsf{V}}^{-1}
\boldsymbol{\mathsf B}(x_n,\boldsymbol{v}_n)
-\boldsymbol{\mathsf{V}}^{-1}\boldsymbol{\mathsf B}
(\mathsf{x},\boldsymbol{\mathsf{v}})\|_{\boldsymbol{\mathsf{V}}}^2
+\widetilde{\zeta}_n(\mathsf{x},\boldsymbol{\mathsf{v}}),
\end{align}
where
\begin{equation}
(\forall n\in\NN)\quad 
\widetilde{\zeta}_n(\mathsf{x},\boldsymbol{\mathsf{v}})=
\|\boldsymbol{\mathsf{V}}^{-1}\|\,\zeta_n
(\mathsf{x},\boldsymbol{\mathsf{v}}).
\end{equation}
According to assumption \ref{a:p3iv}, 
$\big(\widetilde{\zeta}_n(\mathsf{x},
\boldsymbol{\mathsf{v}})\big)_{n\in\NN}\in
\ell^\infty_+({\mathscr{X}})$, and
$\big(\lambda_n\widetilde{\zeta}_n
(\mathsf{x},\boldsymbol{\mathsf{v}})\big)_{n\in\NN}\in
\ell_+^{1/2}({\mathscr{X}})$.
Now, let $n\in\NN$, let 
$({\mathsf{x}},\boldsymbol{\mathsf{v}})\in\KKK$, and set 
$\widetilde{\mathsf{y}}=\prox_{\mathsf{f}}^{\mathsf{W}^{-1}}
(\mathsf{x}-\mathsf{W}\boldsymbol{\mathsf{L}}^*
\boldsymbol{\mathsf{v}})$. 
By \eqref{e:beforemoulinex} and the nonexpansiveness of 
$\prox_{\boldsymbol{\mathsf{g}}^{*}}^{\boldsymbol{\mathsf{U}}^{-1}}$
in $(\GGG,\|\cdot\|_{\boldsymbol{\mathsf{U}}^{-1}})$, we obtain
\begin{align}
&\|\mathsf{J}_{\boldsymbol{\mathsf{V}}^{-1}
\boldsymbol{\mathsf{A}}_n}({\mathsf{x}},
\boldsymbol{\mathsf{v}})-\mathsf{J}_{\boldsymbol{\mathsf{V}}^{-1}
\boldsymbol{\mathsf{A}}}({\mathsf{x}},
\boldsymbol{\mathsf{v}})\|_{\boldsymbol{\mathsf{V}}}^2\nonumber\\
&\leq\|\boldsymbol{\mathsf{V}}\|
\big(\|\mathsf{y}-\widetilde{\mathsf{y}}\|^2+\big\|
\prox_{\boldsymbol{\mathsf{g}}^{*}}^{\boldsymbol{\mathsf{U}}^{-1}}
(\boldsymbol{\mathsf{v}}+\boldsymbol{\mathsf{U}}
\boldsymbol{\mathsf{L}}(2\mathsf{y}-\mathsf{x}))
-\prox_{\boldsymbol{\mathsf{g}}^{*}}^{\boldsymbol{\mathsf{U}}^{-1}}
(\boldsymbol{\mathsf{v}}+\boldsymbol{\mathsf{U}}
\boldsymbol{\mathsf{L}}(2\widetilde{\mathsf{y}}-\mathsf{x}))
\big\|^2\big)
\nonumber\\
&\leq\|\boldsymbol{\mathsf{V}}\|
\big(\|\mathsf{y}-\widetilde{\mathsf{y}}\|^2+4 
\|\boldsymbol{\mathsf{U}}\boldsymbol{\mathsf{L}}
(\mathsf{y}-\widetilde{\mathsf{y}})
\|_{\boldsymbol{\mathsf{U}}^{-1}}^2\big)
\nonumber\\
&\leq\|\boldsymbol{\mathsf{V}}\|
(1+4\|\boldsymbol{\mathsf{U}}\|
\|\boldsymbol{\mathsf{L}}\|^2)
\|\mathsf{y}-\widetilde{\mathsf{y}}\|^2.
\end{align}
It follows from \ref{a:p3v} that
\begin{align}
\label{e:beforemoulinex2}
&\|\mathsf{J}_{\boldsymbol{\mathsf{V}}^{-1}
\boldsymbol{\mathsf{A}}_n}({\mathsf{x}},\boldsymbol{\mathsf{v}}) 
- \mathsf{J}_{\boldsymbol{\mathsf{V}}^{-1}
\boldsymbol{\mathsf{A}}}({\mathsf{x}},\boldsymbol{\mathsf{v}})
  \|_{\boldsymbol{\mathsf{V}}}\nonumber\\
&\le
\|\boldsymbol{\mathsf{V}}\|^{1/2}\|(1+2\|\boldsymbol{\mathsf{U}}\|^{1/2}
\|\boldsymbol{\mathsf{L}}\|)
\|\prox_{\mathsf{f}_n}^{\mathsf{W}^{-1}}
(\mathsf{x}-\mathsf{W}\boldsymbol{\mathsf{L}}^*
\boldsymbol{\mathsf{v}})-
\prox_{\mathsf{f}}^{\mathsf{W}^{-1}}(\mathsf{x}-\mathsf{W}
\boldsymbol{\mathsf{L}}^*\boldsymbol{\mathsf{v}})\|\nonumber\\
&\leq \|\boldsymbol{\mathsf{V}}\|^{1/2}\|
(1+2\|\boldsymbol{\mathsf{U}}\|^{1/2}
\|\boldsymbol{\mathsf{L}}\|)
(\alpha_n
\|\mathsf{x}-\mathsf{W}\boldsymbol{\mathsf{L}}^*
\boldsymbol{\mathsf{v}}\|+\beta_n)\nonumber\\
&\leq \|\boldsymbol{\mathsf{V}}\|^{1/2}\|
(1+2\|\boldsymbol{\mathsf{U}}\|^{1/2}
\|\boldsymbol{\mathsf{L}}\|)
\big(\alpha_n
(\|\mathsf{x}\|+\|\mathsf{W}\boldsymbol{\mathsf{L}}^*\|
\|\boldsymbol{\mathsf{v}}\|)+\beta_n\big)\nonumber\\
&\le \widetilde{\alpha}_n
\|({\mathsf{x}},\boldsymbol{\mathsf{v}})
\|_{\boldsymbol{\mathsf{V}}}+\widetilde{\beta}_n,
\end{align}
where 
\begin{equation}
\begin{cases}
\widetilde{\alpha}_n=\sqrt{2}
\|\boldsymbol{\mathsf{V}}\|^{1/2}
\|(1+2\|\boldsymbol{\mathsf{U}}\|^{1/2}
\|\boldsymbol{\mathsf{L}}\|)
\max\{1,\|\mathsf{W}\boldsymbol{\mathsf{L}}^*\|\}
\|\boldsymbol{\mathsf{V}}^{-1}\|^{1/2} \alpha_n\\
\widetilde{\beta}_n = \|\boldsymbol{\mathsf{V}}\|^{1/2}
\|(1+2\|\boldsymbol{\mathsf{U}}\|^{1/2}
\|\boldsymbol{\mathsf{L}}\|) \beta_n.
\end{cases}
\end{equation}
Thus, $\sum_{n\in\NN}\sqrt{\lambda_n}\widetilde{\alpha}_{n}<\pinf$ 
and $\sum_{n\in\NN}\lambda_n \widetilde{\beta}_{n} < \pinf$.
Finally, since $\gamma_n\equiv 1$, \eqref{e:CentraleP} implies
that $\sup_{n\in\NN}(1+\widetilde{\tau_n})\gamma_n<2\vartheta$.
All the assumptions of Proposition~\ref{p:23} 
are therefore satisfied for algorithm \eqref{e:FBp2}.
\end{proof}
\newpage
\begin{remark} \
\begin{enumerate}
\item 
Algorithm~\ref{e:PDcoordopt1} can be viewed as a stochastic 
version of the primal-dual algorithm investigated in 
\cite[Example~6.4]{Opti14} when the metric is fixed in the latter.
Particular cases of such fixed metric primal-algorithm can be found
in \cite{Cham11,Icip14,Esse10,Heyu12,Komo15}.
\item 
The same type of primal-dual algorithm is investigated in 
\cite{Bian14,Repe15} in a different context since in those
papers the stochastic nature of the algorithms stems from 
the random activation of blocks of variables.
\end{enumerate}
\end{remark}

\subsection{Example}
We illustrate an implementation of Algorithm~\ref{algo:7} in 
a simple scenario with $\HH=\RR^N$ by constructing an example in
which the gradient approximation conditions are fulfilled.

For every $k\in \{1,\ldots,q\}$ and every $n\in\NN$, set
$s_{k,n}=\nabla \mathsf{j}_k^*(v_{k,n})$ and suppose that
$(y_n)_{n\in\NN}$ is almost surely bounded.
This assumption is satisfied, in particular, if
$\dom \mathsf{f}$ and $(b_n)_{n\in\NN}$ are bounded.
In addition, let 
\begin{equation}
(\forall n\in\NN)\quad\boldsymbol{\XX}_n=
\sigma\big(x_0,\boldsymbol{v}_0,(K_{n'},z_{n'})_{0\leq n'<m_n},
(b_{n'},\boldsymbol{c}_{n'})_{1\leq n'< n}\big),
\end{equation}
where $(m_n)_{n\in\NN}$ is a strictly increasing sequence in $\NN$
such that $m_n=O(n^{1+\delta})$ with 
$\delta\in\RPP$, $(K_n)_{n\in\NN}$ is a sequence of independent and
identically distributed 
(i.i.d.) random matrices of $\RR^{M\times N}$,
and $(z_n)_{n\in\NN}$ is a sequence of i.i.d.\ random vectors of
$\RR^M$. For example, in signal recovery, $(K_n)_{n\in\NN}$ may model
a stochastic degradation operators \cite{Comb89}, 
while $(z_n)_{n\in\NN}$ are observations
related to an unknown signal that we want to estimate.
The variables $(K_n,z_n)_{n\in\NN}$ are supposed to be
independent of
$(b_{n},\boldsymbol{c}_{n})_{n\in\NN}$
and such that $\EE\|K_0\|^4<\pinf$ and $\EE\|z_0\|^4<\pinf$. Set
\begin{equation}
(\forall \mathsf{x}\in \HH)\quad \mathsf{h}(\mathsf{x})=\frac12
\EE\|K_0\mathsf{x}-z_0 \|^2
\end{equation}
and, for every $n\in\NN$, let
\begin{equation}
u_n=\frac{1}{m_{n+1}}\sum_{n'=0}^{m_{n+1}-1} 
K_{n'}^\top (K_{n'} x_n-z_{n'})
\end{equation}
be an empirical estimate of $\nabla\mathsf{h}(x_n)$. 
We assume that $\lambda_n=O(n^{-\kappa})$ where 
$\kappa\in\left]1-\delta,1\right]\cap [0,1]$. We have
\begin{equation}
\label{e:ECunahxQ}
(\forall n \in\NN)\quad
\EC{u_n}{\boldsymbol{\XX}_n}-\nabla \mathsf{h}(x_{n}) 
=\frac{1}{m_{n+1}} \big(Q_{0,m_{n}} x_n-r_{0,m_{n}}\big)
\end{equation}
where, for every $(n_1,n_2)\in\NN^2$ such that $n_1<n_2$,
\begin{equation}
Q_{n_1,n_2} = \sum_{n'=n_1}^{n_2-1} \big(K_{n'}^\top
K_{n'}-\EE(K_0^\top K_0)\big)
\quad\text{and}\quad
r_{n_1,n_2} = \sum_{n'=n_1}^{n_2-1} \big(K_{n'}^\top
z_{n'}-\EE(K_0^\top z_0)\big).
\end{equation}
From the law of iterated logarithm \cite[Section 25.8]{Davi94}, we
have almost surely
\begin{equation}
\varlimsup_{n\to\pinf} \frac{\|Q_{0,m_{n}} \|}{\sqrt{m_{n}
\log(\log(m_{n}))}}<\pinf
\quad\text{and}\quad
\varlimsup_{n\to \pinf} \frac{\|r_{0,m_{n}} \|}{\sqrt{m_{n}
\log(\log(m_{n}))}}< \pinf.
\end{equation}
Since $(y_n)_{n\in\NN}$ is assumed to be bounded,
there exists a $\RP$-valued random variable $\eta$ such that, 
for every $n\in\NN$, $\sup_{n\in\NN}\|y_n\|\leq\eta$.
Therefore, 
\begin{equation}
\label{e:xnboundQ}
(\forall n \in\NN)\quad \|x_n\| \leq \|x_0\|+\eta.
\end{equation}
Altogether, \eqref{e:ECunahxQ}--\eqref{e:xnboundQ} yield
\begin{equation}
\lambda_n\|\EC{u_n}{\boldsymbol{\XX}_n}-\nabla
\mathsf{h}(x_{n})\|^2 
=O\Big(\Frac{\lambda_n m_{n}\log(\log(m_{n}))}{m_{n+1}^2} \Big)=
O\Big(\Frac{\log(\log(n))}{n^{1+\delta+\kappa}} \Big).
\end{equation}
Consequently, assumption \ref{a:p3ii} in 
Proposition~\ref{p:3} holds. In addition, for every $n\in\NN$,
\begin{equation}
u_n-\EC{u_n}{\boldsymbol{\XX}_n}
=\frac{1}{m_{n+1}}\big(Q_{m_n,m_{n+1}}x_n-r_{m_n,m_{n+1}}\big)
\end{equation}
which, by the triangle inequality, implies that
\begin{align}
\EC{\|u_n-\EC{u_n}{\boldsymbol{\XX}_n}\|^2}{\boldsymbol{\XX}_n}
&\leq\frac{1}{m_{n+1}^2} \EC{(\|Q_{m_n,m_{n+1}}\|\,
\|x_n\|+\|r_{m_n,m_{n+1}}\|)^2}{\boldsymbol{\XX}_n}\nonumber\\
&\leq\frac{2}{m_{n+1}^2} \big(\EE\|Q_{m_n,m_{n+1}}\|^2\,
\|x_n\|^2+\EE\|r_{m_n,m_{n+1}}\|^2\big).
\end{align}
Upon invoking the i.i.d. assumptions, we obtain
\begin{equation}
(\forall n \in\NN)\qquad 
\begin{cases}
\EE\|Q_{m_n,m_{n+1}}\|^2=(m_{n+1}-m_n)\EE\|K_{0}^\top
K_{0}-\EE(K_0^\top K_0)\|^2\\
\EE\|r_{m_n,m_{n+1}}\|^2=(m_{n+1}-m_n)\EE\|K_{0}^\top
z_{0}-\EE(K_0^\top z_0)\|^2
\end{cases}
\end{equation}
and it therefore follows from \eqref{e:xnboundQ} that
\begin{equation}
\zeta_n = \EC{\|u_n-\EC{u_n}{\boldsymbol{\XX}_n}\|^2}
{\boldsymbol{\XX}_n}=O\Big(\frac{m_{n+1}-m_n}{m_{n+1}^2}
\Big)=O\Big(\frac{1}{n^{2+\delta}}\Big)
\end{equation}
and
\begin{equation}
\lambda_n\zeta_n=O\Big(\frac{1}{n^{2+\delta+\kappa}}\Big).
\end{equation}
Thus, assumption~\ref{a:p3iv} in Proposition~\ref{p:3} 
holds with $\tau_n\equiv 0$.


\newpage
\begin{thebibliography}{99}
\setlength{\itemsep}{1pt} 

\small

\bibitem{Atch14}
Y.~F. Atchad\'e, G.~Fort, and E.~Moulines, 
On stochastic proximal gradient algorithms, 
2014.\\
{\ttfamily{http://arxiv.org/abs/1402.2365}}

\bibitem{Sico10}  
H. Attouch, L. M. Brice\~no-Arias, and P. L. Combettes,
A parallel splitting method for coupled monotone inclusions,
{\em SIAM J. Control Optim.},
vol. 48, pp. 3246--3270, 2010. 

\bibitem{Bach11}
F. Bach and E. Moulines, 
Non-asymptotic analysis of stochastic approximation algorithms for 
machine learning, in 
\emph{Proc. Ann. Conf. Neur. Inform. Proc. Syst.}, 
Granada, Spain, pp. 451--459, 2011.

\bibitem{Livre1} 
H. H. Bauschke and P. L. Combettes,
{\em Convex Analysis and Monotone Operator Theory in Hilbert 
Spaces.}
Springer, New York, 2011.

\bibitem{Bian14}
P.~Bianchi, W.~Hachem, and F.~Iutzeler, 
A stochastic coordinate descent primal-dual algorithm and 
applications to large-scale composite optimization, 2014.\\
{\ttfamily{http://arxiv.org/abs/1407.0898}}

\bibitem{Botr15}
R. I. Bo{\c{t}} and  E. R. Csetnek,
On the convergence rate of a forward-backward type primal-dual
splitting algorithm for convex optimization problems,
{\em Optimization,}
vol. 64, pp. 5--23, 2015.

\bibitem{Bric13} 
L. M. Brice\~{n}o-Arias and P. L. Combettes, 
Monotone operator methods for Nash equilibria in non-potential 
games, in
\emph{Computational and Analytical Mathematics,} 
(D. Bailey {\em et. al.}, eds.), pp. 143--159.
Springer, New York, 2013.

\bibitem{Jmiv11}
L. M. Brice\~{n}o-Arias, P. L. Combettes, J.-C. Pesquet, and
N. Pustelnik,
Proximal algorithms for multicomponent image recovery 
problems,
{\em J. Math. Imaging Vision,}
vol. 41, pp. 3--22, 2011.

\bibitem{Byrn14} 
C. L. Byrne,
{\em Iterative Optimization in Inverse Problems.}
CRC Press, Boca Raton, FL, 2014.

\bibitem{Chaa11} 
L. Cha\^ari, J.-C. Pesquet, A. Benazza-Benyahia, and 
P. Ciuciu, A wavelet-based regularized reconstruction 
algorithm for SENSE parallel MRI with applications to neuroimaging,
{\em Med. Image Anal.,}
vol. 15, pp. 185--201, 2011.

\bibitem{Cham15}
A. Chambolle and C. Dossal,
On the convergence of the iterates of the ``fast iterative 
shrinkage/thresholding algorithm,"
{\em J. Optim. Theory Appl.,} 2015.

\bibitem{Cham11}
A. Chambolle and T. Pock, 
A first-order primal-dual algorithm for convex problems with 
applications to imaging, 
{\em J. Math. Imaging Vision,}
vol. 40, pp. 120--145, 2011.

\bibitem{Chau07}
C. Chaux, P. L. Combettes, J.-C. Pesquet, and V. R. Wajs,
A variational formulation for frame-based inverse problems,
{\em Inverse Problems,}
vol. 23, pp. 1495--1518, 2007.

\bibitem{Opti04}
P. L. Combettes, Solving monotone inclusions via 
compositions of nonexpansive averaged operators,
{\em Optimization,}
vol. 53, pp. 475--504, 2004.

\bibitem{Svva10} 
P. L. Combettes, D\hspace{-1.6ex}\raise 0.4ex
\hbox{-}\hspace{.8ex}inh D\~ung, and B. C. V\~u, 
Dualization of signal recovery problems,
{\em Set-Valued Anal.,} 
vol. 18, pp. 373--404, 2010.

\bibitem{Icip14}
P. L Combettes, L. Condat, J.-C. Pesquet, and B. C. V\~u.
A forward-backward view of some primal-dual optimization 
methods in image recovery,
{\em Proc. IEEE Int. Conf. Image Process.}, 
pp. 4141--4145. Paris, France, 2014.

\bibitem{Svva12}
P. L. Combettes and J.-C. Pesquet, 
Primal-dual splitting algorithm for solving
inclusions with mixtures of composite, Lipschitzian, and 
parallel-sum type monotone operators,
{\em Set-Valued Var. Anal.},
vol. 20, pp. 307--330, 2012.

\bibitem{Siop15} 
P. L. Combettes and J.-C. Pesquet, 
Stochastic quasi-Fej\'er block-coordinate fixed
point iterations with random sweeping,
{\em SIAM J. Optim.,} 
vol. 25, pp. 1221--1248, 2015.

\bibitem{Comb89} 
P. L. Combettes and H. J. Trussell, 
Methods for digital restoration of signals degraded by a 
stochastic impulse response,
{\em IEEE Trans. Acoustics, Speech, Signal Process., }
vol. 37, pp. 393--401, 1989. 

\bibitem{Opti14}
P. L. Combettes and B. C. V\~u,
Variable metric forward-backward splitting with applications 
to monotone inclusions in duality,
{\em Optimization,}
vol. 63, pp. 1289--1318, 2014.

\bibitem{Smms05}
P. L. Combettes and V. R. Wajs, 
Signal recovery by proximal forward-backward splitting,
{\em Multiscale Model. Simul.,} 
vol. 4, pp. 1168--1200, 2005.

\bibitem{Yama15}
P. L. Combettes and I. Yamada, 
Compositions and convex combinations of averaged nonexpansive
operators,
{\em J. Math. Anal. Appl.,}
vol. 425, pp. 55--70, 2015.

\bibitem{Cond13}
L. Condat, 
A primal-dual splitting method for convex optimization involving
Lipschitzian, proximable and linear composite terms,
{\em J. Optim. Theory Appl.,}
vol. 158, pp. 460--479, 2013. 

\bibitem{Davi94}
J. Davidson, 
\emph{Stochastic Limit Theory}. 
Oxford University Press, New York, 1994.

\bibitem{Devi11}
E. De Vito, V. Umanit\`a, and S. Villa,
A consistent algorithm to solve Lasso, elastic-net and
Tikhonov regularization,
{\em J. Complexity,}
vol. 27, pp. 188--200, 2011.

\bibitem{Duch09}
J. Duchi and Y. Singer, 
Efficient online and batch learning using forward backward 
splitting,
{\em J. Mach. Learn. Res.,}
vol. 10, pp. 2899--2934, 2009.

\bibitem{Ermo69} 
Yu. M. Ermol'ev, 
On the method of generalized stochastic gradients and 
quasi-Fej\'er sequences,
{\em Cybernetics,} 
vol. 5, pp. 208--220, 1969.

\bibitem{Ermo66}
Yu. M. Ermoliev and Z. V. Nekrylova, 
Some methods of stochastic optimization,
{\em Kibernetika (Kiev),} 
vol. 1966, pp. 96--98, 1966.

\bibitem{Ermo67}
Yu. M. Ermoliev and Z. V. Nekrylova, 
The method of stochastic gradients and its application, 
in {\em Seminar: Theory of Optimal Solutions,} no. 1, 
{\em Akad. Nauk Ukrain. SSR, Kiev,} pp. 24--47, 1967.

\bibitem{Esse10}
E.~Esser, X.~Zhang, and T.~Chan, 
A general framework for a class of first order primal-dual 
algorithms for convex optimization in imaging science,
\emph{SIAM J. Imaging Sci.}, vol.~3, no.~4, pp. 1015--1046, 2010.

\bibitem{Facc03} 
F. Facchinei and J.-S. Pang,
{\em Finite-Dimensional Variational Inequalities and 
Complementarity Problems.}
Springer-Verlag, New York, 2003.

\bibitem{Fort95}
R. M. Fortet,
{\em Vecteurs, Fonctions et Distributions Al\'eatoires dans les
Espaces de Hilbert.}
Herm\`es, Paris, 1995.

\bibitem{Guse71}
O. V. Guseva, 
The rate of convergence of the method of generalized stochastic
gradients, 
{\em Kibernetika (Kiev),} 
vol. 1971, pp. 143--145, 1971.

\bibitem{Heyu12}
B. He and X. Yuan, 
Convergence analysis of primal-dual algorithms for a
saddle-point problem: from contraction perspective, 
\emph{SIAM J. Imaging Sci.}, 
vol. 5, pp. 119--149, 2012.

\bibitem{Komo15}
N. Komodakis and J.-C. Pesquet, 
Playing with duality: An overview of recent primal-dual approaches 
for solving large-scale optimization problems,
{\em IEEE Signal Process. Mag.}, to appear.
{\ttfamily{http://www.optimization-online.org/DB\_HTML/2014/06/4398.html}}

\bibitem{Kush03}
H. J. Kushner and G. G. Yin, 
{\em Stochastic Approximation and Recursive Algorithms and 
Applications}, 2nd ed.
Springer-Verlag, New York, 2003.

\bibitem{Ledo91}
M. Ledoux and M. Talagrand,
{\em Probability in Banach Spaces: Isoperimetry and Processes}.
Springer, New York, 1991.

\bibitem{Lema96} 
B. Lemaire,
Stability of the iteration method for nonexpansive mappings,
{\em Serdica Math. J.},
vol. 22, pp. 331--340, 1996.

\bibitem{Lema97} 
B. Lemaire,
Which fixed point does the iteration method select?
{\em Lecture Notes in Economics and Mathematical Systems},
vol. 452, pp. 154--167. Springer-Verlag, New York, 1997.

\bibitem{Merc79} 
B. Mercier, 
{\em Topics in Finite Element Solution of Elliptic Problems}
(Lectures on Mathematics, no. 63).
Tata Institute of Fundamental Research, Bombay, 1979.

\bibitem{Merc80} 
B. Mercier,  
{\em In\'equations Variationnelles de la M\'ecanique}
(Publications Math\'ematiques d'Orsay, no. 80.01).
Universit\'e de Paris-XI, Orsay, France, 1980. 

\bibitem{Nekr74}
Z. V. Nekrylova, 
Solution of certain variational problems and control problems by 
the stochastic gradient method, 
{\em Kibernetika (Kiev),}
vol. 1974, pp. 62--66, 1974; translated in 
{\em Cybernetics,}
vol. 10, pp. 622--626, 1976.

\bibitem{Repe15}
J.-C. Pesquet and A. Repetti, 
A class of randomized primal-dual algorithms for distributed 
optimization, 
\emph{J. Nonlinear Convex Anal.}, to appear.
{\ttfamily{http://arxiv.org/abs/1406.6404}}

\bibitem{Robi51}
H. Robbins and S. Monro, A stochastic approximation method, 
\emph{Ann. Math. Statistics}, 
vol. 22, pp. 400--407, 1951.

\bibitem{Ros14b}
L. Rosasco, S. Villa, and B. C. V\~u, 
Convergence of stochastic proximal gradient algorithm, 2014.
{\ttfamily{http://arxiv.org/abs/1403.5074}}

\bibitem{Ros14a}
L. Rosasco, S. Villa, and B. C. V\~u, 
A stochastic forward-backward splitting method for solving monotone
inclusions in {H}ilbert spaces, 2014. 
{\ttfamily{http://arxiv.org/abs/1403.7999}}

\bibitem{Rosa15}
L. Rosasco, S. Villa, and B. C. V\~u, 
A stochastic inertial forward-backward splitting algorithm for 
multivariate monotone inclusions, 2015. 
{\ttfamily{http://arxiv.org/abs/1507.00848}}

\bibitem{SS2013}
S. Shalev-Shwartz and T. Zhang, 
Stochastic dual coordinate ascent methods for regularized loss 
minimization, 
\emph{J. Mach. Learn. Res.}, 
vol. 14, pp. 567--599, 2013.

\bibitem{Shor85} N. Z. Shor, 
{\em Minimization Methods for Non-Differentiable Functions.} 
Springer-Verlag, New York, 1985.

\bibitem{Tali15}
C. Talischi and G. H. Paulino,
A closer look at consistent operator splitting and its extensions for
topology optimization,
{\em Comput. Methods Appl. Mech. Engrg.,}
vol. 283, pp. 573--598, 2015.

\bibitem{Tsen90} 
P. Tseng, Further applications of a splitting algorithm to
decomposition in variational inequalities and convex programming,
{\em Math. Programming,}
vol. 48,  pp. 249--263, 1990.

\bibitem{Tsen91} 
P. Tseng, Applications of a splitting algorithm to
decomposition in convex programming and variational inequalities,
{\em SIAM J. Control Optim.,}
vol. 29, pp. 119--138, 1991.

\bibitem{Bang13}
B. C. V\~u, A splitting algorithm for dual monotone inclusions 
involving cocoercive operators, 
{\em Adv. Comput. Math.}, 
vol. 38, pp. 667--681, 2013.

\bibitem{Widr03}
B.~Widrow and S.~D. Stearns, 
\emph{Adaptive Signal Processing}. 
Prentice-Hall, Englewood Cliffs, NJ, 1985.

\bibitem{Xiao14}
L.~Xiao and T.~Zhang, 
A proximal stochastic gradient method with progressive
variance reduction, 
\emph{SIAM J. Optim.}, vol.~24, pp. 2057--2075,
  2014.
\end{thebibliography}
\end{document}